\documentclass[12pt]{article}   
\pdfoutput=1
\usepackage{fullpage}
\usepackage{pinlabel}
\usepackage[all]{xy}

\title{Orbifold points on Teichm\"uller curves and \\ Jacobians with complex multiplication}
\date{August 28, 2013}

\author{Ronen E. Mukamel\footnote{The author is partially supported by award DMS-1103654 from the National Science Foundation.}} 

\usepackage{amsmath}
\usepackage[small,sf]{caption}
\usepackage{verbatim}
\usepackage{booktabs}
\usepackage{graphicx}
\usepackage{amsfonts}
\usepackage{amsthm}
\usepackage{ifpdf}
\usepackage{ifvtex}

\newtheorem{thm}{Theorem}[section]
\newtheorem{prop}[thm]{Proposition}
\newtheorem{pro}{Problem}

\newtheorem{cor}[thm]{Corollary}

\theoremstyle{remark}
\newtheorem*{rmk}{Remark}
\numberwithin{equation}{section}

\def\co{\colon\thinspace}
\def \als {\addlinespace[0.75em]}
\newcommand{\appwdtop}{B}
\newcommand{\apporbptmodels}{C}

\newcommand{\spmat}[1]{\left(\begin{smallmatrix} #1 \end{smallmatrix}\right)}

\newcommand{\Out}{\operatorname{Out}}
\newcommand{\PSO}{\operatorname{PSO}}
\newcommand{\PSL}{\operatorname{PSL}}
\newcommand{\Arf}{\operatorname{Arf}}
\newcommand{\Stab}{\operatorname{Stab}}
\newcommand{\chat}{\widehat \cc}

\newcommand{\hr}{\widetilde h}
\newcommand{\lsymb}[2]{\left(\frac{#1}{#2} \right)}
\newcommand{\Fix}{\operatorname{Fix}}

\newcommand{\nm}{\operatorname{Nm}}

\newcommand{\Sym}{\operatorname{Sym}}

\newcommand{\End}{\operatorname{End}}
\newcommand{\Aut}{\operatorname{Aut}}

\newcommand{\Tr}{\operatorname{Tr}}

\newcommand{\teich}{Teichm\"uller}
\newcommand{\wt}{\widetilde}
\newcommand{\order}{\mathcal O}
\newcommand{\ord}{\mathcal O}
\newcommand{\jac}{\operatorname{Jac}}

\newcommand{\abs}[1]{\left|#1\right|}
\newcommand{\vect}[2]{\left( \begin{smallmatrix} #1 \\ #2 \end{smallmatrix} \right)}
\newcommand{\cc}{\mathbb C}
\newcommand{\hh}{\mathbb H}
\newcommand{\rr}{\mathbb R}
\newcommand{\qq}{\mathbb Q}
\newcommand{\zz}{\mathbb Z}

\newcommand{\Pic}{\operatorname{Pic}}
\newcommand{\CM}{\operatorname{CM}}
\newcommand{\mosp}{\mathcal M}
\newcommand{\M}{\mathcal M}

\newcommand{\SL}{\operatorname{SL}}

\newcommand{\strfam}{\ensuremath{\mathcal Z_4}}

\renewcommand{\Im}{\operatorname{Im}}
\renewcommand{\Re}{\operatorname{Re}}

\begin{document}
\maketitle
%
%
%
\begin{abstract}
For each integer $D \geq 5$ with $D \equiv 0$ or $1 \bmod 4$, the Weierstrass curve $W_D$ is an algebraic curve and a finite volume hyperbolic orbifold which admits an algebraic and isometric immersion into the moduli space of genus two Riemann surfaces.  The Weierstrass curves are the main examples of Teichm\"uller curves in genus two.  The primary goal of this paper is to determine the number and type of orbifold points on each component of $W_D$.   Our enumeration of the orbifold points, together with \cite{bainbridge:eulerchar} and \cite{mcmullen:spin}, completes the determination of the homeomorphism type of $W_D$ and gives a formula for the genus of its components.  We use our formula to give bounds on the genus of $W_D$ and determine the Weierstrass curves of genus zero.  We will also give several explicit descriptions of each surface labeled by an orbifold point on $W_D$.
\end{abstract}

\section{Introduction}
\label{sec:introduction}
Let $\M_g$ be the moduli space of genus $g$ Riemann surfaces.  The space $\M_g$ can be viewed as both a complex orbifold and an algebraic variety and carries a complete Finsler \teich\ metric.  A {\em \teich\ curve} is an algebraic and isometric immersion of a finite volume hyperbolic Riemann surface into moduli space:
\[ f\co C=\hh/\Gamma \rightarrow \M_g. \]
The modular curve $\hh/\SL_2(\zz) \rightarrow \M_1$ is the first example of a \teich\ curve.  Other examples emerge from the study of polygonal billiards \cite{veech:ngons,masurtabachnikov:billiards} and square-tiled surfaces.  While the \teich\ curves in $\M_2$ have been classified \cite{mcmullen:torsion}, much less is known about \teich\ curves in $\M_g$ for $g > 2$ \cite{bainbridgemoeller:genus3,bouwmoeller:triangleveechgps}.

The main source of Teichm\"uller curves in $\M_2$ are the Weierstrass curves.  For each integer $D \geq 5$ with $D \equiv 0$ or $1 \bmod 4$, the {\em Weierstrass curve} $W_D$ is the moduli space of Riemann surfaces whose Jacobians have real multiplication by the quadratic order $\ord_D=\zz\left[ \frac{D+\sqrt{D}}{2} \right]$ stabilizing a holomorphic one form with double zero up to scale.  The curve $W_D$ is a finite volume hyperbolic orbifold and the natural immersion
\[ W_D \rightarrow \M_2 \]
is algebraic and isometric and has degree one onto its image \cite{calta:periodicity,mcmullen:billiards}.  The curve $W_D$ is a \teich\ curve unless $D > 9$ with $D \equiv 1 \bmod 8$ in which case $W_D = W_D^0 \sqcup W_D^1$ is a disjoint union of two \teich\ curves distinguished by a spin invariant in $\zz/2\zz$ \cite{mcmullen:spin}.  A major challenge is to describe $W_D$ as an algebraic curve and as a hyperbolic orbifold.  To date, this has been accomplished only for certain small $D$ \cite{bouwmoeller:nonarithmetic,mcmullen:billiards,lochak:arithmetic}.  

The purpose of this paper is to study the orbifold points on $W_D$.  Such points label surfaces with automorphisms commuting with $\ord_D$.  The first two Weierstrass curves $W_5$ and $W_8$ were studied by Veech \cite{veech:ngons} and are isomorphic to the $(2,5,\infty)$- and $(4,\infty,\infty)$-orbifolds.  The surfaces with automorphisms labeled by the three orbifold points are drawn in Figure \ref{fig:ngonorbifoldpts}. 

Our primary goal is to give a formula for the number and type of orbifold points on $W_D$ (Theorem \ref{thm:orbformula}).  Together with \cite{mcmullen:spin} and \cite{bainbridge:eulerchar}, our formula completes the determination of the homeomorphism type of $W_D$ and gives a formula for the genus of $W_D$.  We will use our formula to give bounds for the genera of $W_D$ and $W_D^\epsilon$ (Corollary \ref{cor:genus}) and list the components of $\bigcup_D W_D$ of genus zero (Corollary \ref{cor:genuszerocomponents}).  We will also give several explicit descriptions of the surfaces labeled by orbifold points on $W_D$ (Theorems \ref{thm:dsurfaces}), giving the first examples of algebraic curves labeled by points of $W_D$ for most $D$ (Theorem \ref{thm:algebraicmodels}). 

\paragraph{Main results.} Our main theorem determines the number and type of orbifold points on $W_D$:
\begin{thm}
  \label{thm:orbformula}
  For $D > 8$, the orbifold points on $W_D$ all have order two, and the number of such points $e_2(W_D)$ is the weighted sum of class numbers of imaginary quadratic orders shown in Table \ref{tab:orbformula}.
\end{thm}
\begin{table}
  \[
  \begin{array}{|c|c|}
    \hline
    D \bmod 16 & e_2(W_D) \\
    \hline 
    \rule{0pt}{15pt}
    1,5,9, \mbox{ or }13 & \frac{1}{2}\hr(-4D)\\
    0 & \frac{1}{2} (\hr(-D)+2\hr(-D/4)) \\
    4 & 0 \\
    8 & \frac{1}{2}\hr(-D) \\
    12 & \frac{1}{2} (\hr(-D)+3\hr(-D/4)) \\
    \hline
  \end{array}
  \] 
  \caption{\label{tab:orbformula} {\sf For $D>8$, the number of orbifold points of order two on $W_D$ is given by a weighted sum of class numbers.  The function $\hr(-D)$ is defined below. }}
\end{table}
We also give a formula for the number of orbifold points on each spin component:
\begin{thm} \label{thm:spin}
  Fix $D\geq 9$ with $D \equiv 1\bmod 8$.  If $D = f^2$ is a perfect square, then all of the orbifold points on $W_D$ lie on the spin $(f+1)/2 \bmod 2$ component:
  \[ e_2\left( W_D^{(f+1)/2} \right)=\frac{1}{2}\hr(-4D) \mbox{ and } e_2\left( W_D^{(f-1)/2} \right)=0.\]
  Otherwise, $e_2(W_D^0)=e_2(W_D^1)=\frac{1}{4}\hr(-4D)$.
\end{thm}
When $D$ is not a square and $W_D$ is reducible, the spin components of $W_D$ have algebraic models defined over $\qq(\sqrt{D})$ and are Galois conjugate \cite{bouwmoeller:nonarithmetic}.  Theorem \ref{thm:spin} confirms that the spin components have the same number and type of orbifold points.

The class number $h(-D)$ is the order of the ideal class group $H(-D)$ for $\order_{-D}$ and counts the number of elliptic curves with complex multiplication by $\order_{-D}$ up to isomorphism.  The weighted class number
\[ \hr(-D) = 2 h(-D)/\abs{\order_{-D}^\times} \]
appearing in Table \ref{tab:orbformula} is the number of elliptic curves with complex multiplication weighted by their orbifold order in $\mosp_1$.  Note that $\hr(-D) = h(-D)$ unless $D=3$ or $4$.  The class number $h(-D)$ can be computed by enumerating integer points on a conic.  We will give a similar method for computing $e_2(W_D)$ in Theorem \ref{thm:enumerateorbpts}.  When $D$ is odd, the orbifold points on $W_D$ are labeled by elements of the group $H(-4D)/[P]$ where $[P]$ is the ideal class in $\ord_{-4D}$ representing the prime ideal with norm two.   

The orbifold Euler characteristics of $W_D$ and $W_D^\epsilon$ were computed in \cite{bainbridge:eulerchar} and the cusps on $W_D$ were enumerated and sorted by component in \cite{mcmullen:spin}.  Theorems \ref{thm:orbformula} and \ref{thm:spin} complete the determination of the homeomorphism type of $W_D$ and give a formula for the genera of $W_D$ and its components.  
\begin{cor} \label{cor:genus}
  For any $\epsilon > 0$, there are constants $C_\epsilon$ and $N_\epsilon$ such that: 
  \[ C_\epsilon D^{3/2+\epsilon} > g(V) > D^{3/2}/650, \]
  whenever $V$ is a component of $W_D$ and $D \geq N_\epsilon$.
\end{cor}
Modular curves of genus zero play an important role in number theory \cite{tits:monster}.  We also determine the components of Weierstrass curves of genus zero. 
\begin{cor} 
  \label{cor:genuszerocomponents}
  The genus zero components of $\bigcup_D W_D$ are the $23$ components of $\bigcup_{D \leq 41} W_D$ and the curves $W_{49}^0$, $W_{49}^1$ and $W_{81}^1$. 
\end{cor}
We include a table listing the homeomorphism type of $W_D$ for $D\leq 225$ in \S \appwdtop.

\begin{figure}
  \begin{center}
    \includegraphics[scale=0.4]{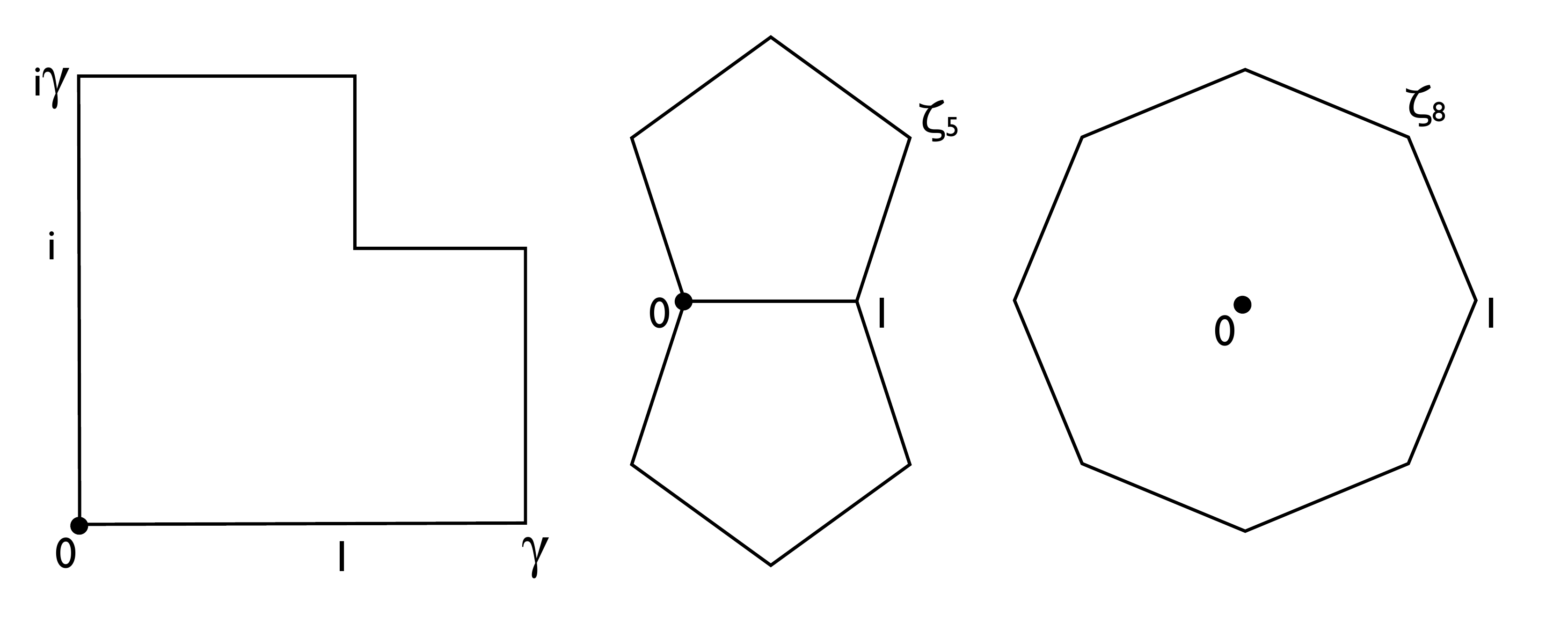}
\end{center}
\caption{\label{fig:ngonorbifoldpts} {\sf The first two Weierstrass curves $W_5$ and $W_8$ are isomorphic to the $(2,5,\infty)$ and $(4,\infty,\infty)$-orbifolds.  The point of orbifold order two is related to billiards on the $L$-shaped table (left) corresponding to the golden mean $\gamma = \frac{1+\sqrt{5}}{2}$.  The points of orbifold order five (center) and four (right) are related to billiards on the regular pentagon and octagon.}
  } 
\end{figure}

\paragraph{Orbifold points on Hilbert modular surfaces.}  Theorem \ref{thm:orbformula} is closely related to the classification of orbifold points on Hilbert modular surfaces we prove in \S \ref{sec:hmsorbpts}.  The {\em Hilbert modular surface} $X_D$ is the moduli space of principally polarized Abelian varieties with real multiplication by $\ord_D$.  The period map sending a Riemann surface to its Jacobian embeds $W_D$ in $X_D$.

Central to the story of the orbifold points on $X_D$ and $W_D$ are the moduli spaces $\M_2(D_8)$ and $\M_2(D_{12})$ of genus two surfaces with actions of the dihedral groups of orders 8 and 12:
\[ D_8 = \left< r, J : r^2 = (Jr)^2 = J^4 = 1 \right> \mbox{ and } D_{12} = \left< r, Z: r^2 = (Zr)^2 = Z^6 = 1 \right>. \]
The surfaces in $\M_2(D_8)$ (respectively $\M_2(D_{12})$) whose Jacobians have complex multiplication have real multiplication commuting with $J$ (respectively $Z$).  The complex multiplication points on $\M_2(D_8)$ and $\M_2(D_{12})$ give most of the orbifold points on $\bigcup_D X_D$:
\begin{thm}  \label{thm:hmsorbpts} 
  The orbifold points on $\bigcup_D X_D$ which are not products of elliptic curves are the two points of order five on $X_5$ and the complex multiplication points on $\M_2(D_8)$ and $\M_2(D_{12})$.
\end{thm}
Since the $Z$-eigenforms on $D_{12}$-surfaces have simple zeros and the $J$-eigenforms on $D_8$-surfaces have double zeros (cf. Proposition \ref{prop:JZeigenforms}), we have:
\begin{cor}
  \label{cor:wdorbpts}
  The orbifold points on $\bigcup_D W_D$ are the point of order five on $W_5$ and the complex multiplication points on $\M_2(D_8)$.  
\end{cor}
Corollary \ref{cor:wdorbpts} explains the appearance of class numbers in the formula for $e_2(W_D)$.  As we will see \S \ref{sec:d8family}, the involution $r$ on a $D_8$-surface $X$ has a genus one quotient $E$ with a distinguished base point and point of order two and the family $\M_2(D_8)$ is birational to the modular curve $Y_0(2)$.  The Jacobian $\jac(X)$ has complex multiplication by an order in $\qq(\sqrt{D},i)$ if and only if $E$ has complex multiplication by an order in $\qq(\sqrt{-D})$.  The formula for $e_2(W_D)$ follows by sorting the $3\hr(-D)$ surfaces with $D_8$-action covering elliptic curves with complex multiplication by $\ord_{-D}$ by their orders for real multiplication.

\paragraph{The product locus $P_D$.}  A recurring theme in the study of the Weierstrass curves is the close relationship between $W_D$ and the product locus $P_D \subset X_D$.  The {\em product locus} $P_D$ consists of products of elliptic curves with real multiplication by $\ord_D$ and is isomorphic to a disjoint union of modular curves.

The cusps on $W_D$ were first enumerated and sorted by spin in \cite{mcmullen:spin} and, for non-square $D$, are in bijection with the cusps on $P_D$.  The Hilbert modular surface $X_D$ has a meromorphic modular form with a simple pole along $P_D$ and a simple zero along $W_D$.  This modular form can be used to give a formula for the Euler characteristic of $W_D$ and, for non-square $D$, the Euler characteristics of $W_D$, $X_D$ and $P_D$ satisfy (\cite{bainbridge:eulerchar}, Cor. 10.4):
\[ \chi(W_D) = \chi(P_D) - 2\chi(X_D). \]
Our classification of the orbifold points on $X_D$ and $W_D$ in Theorem \ref{thm:hmsorbpts} and Corollary \ref{cor:wdorbpts} show that all of the orbifold points of order two on $X_D$ lie on $W_D$ or $P_D$, giving:
\begin{thm}
  For non-square $D$, the homeomorphism type of $W_D$ is determined by the homeomorphism types of $X_D$ and $P_D$ and $D \bmod 8$.
\end{thm}

\paragraph{The $D_8$-family.}  A secondary goal of our analysis is to give several explicit descriptions of $D_8$-surfaces and to characterize those with complex multiplication.  We now outline the facts about $\M_2(D_8)$ which we prove in Section \ref{sec:d8family}; we will outline a similar discussion for $\M_2(D_{12})$ in \S \ref{sec:d12family}.  For a genus two surface $X \in \M_2$, the following are equivalent:
\begin{enumerate}
  \item {\em Automorphisms.}  The automorphism group $\Aut(X)$ admits an injective homomorphism $\rho \co D_8 \rightarrow \Aut(X)$.
  \item {\em Algebraic curves.} The field of meromorphic functions $\cc(X)$ is isomorphic to:
    \[ K_a = \cc(z,x) \mbox{ with } z^2=(x^2-1)(x^4-ax^2+1), \]
    for some $a \in \cc \setminus \left\{ \pm 2 \right\}$.  
  \item {\em Jacobians.}  There is a number $\tau \in \hh$ such that the Jacobian $\jac(X)$ is isomorphic to the principally polarized Abelian variety:
    \[ A_\tau = \cc^2 / \Lambda_\tau, \]
    where $\Lambda_\tau = \zz \left< \vect{\tau}{\tau+1},\vect{\tau}{-\tau-1}, \vect{\tau+1}{\tau},\vect{\tau+1}{-\tau}\right>$ and $A_\tau$ is polarized by the symplectic form $\left< \vect{a}{b},\vect{c}{d}\right>=\frac{-\Im(a \overline c+b \overline d)}{2 \Im(\tau)}$.  
  \item {\em Pinwheels.}  The surface $X$ is isomorphic to the surface $X_\tau$ obtained from the polygonal {\em pinwheel} $P_\tau$ (Figure \ref{fig:pinwheels}) for some $\tau$ in the domain:
    \[ U = \left\{ \tau \in \hh : \tau \neq \frac{\pm 1+i}{2}, \abs{\tau}^2 \geq \frac{1}{2} \mbox{ and } \abs{\Re{\tau}} \leq \frac{1}{2} \right\}. \]
\end{enumerate}
\begin{figure}
  \begin{center}
    \includegraphics[scale=0.4]{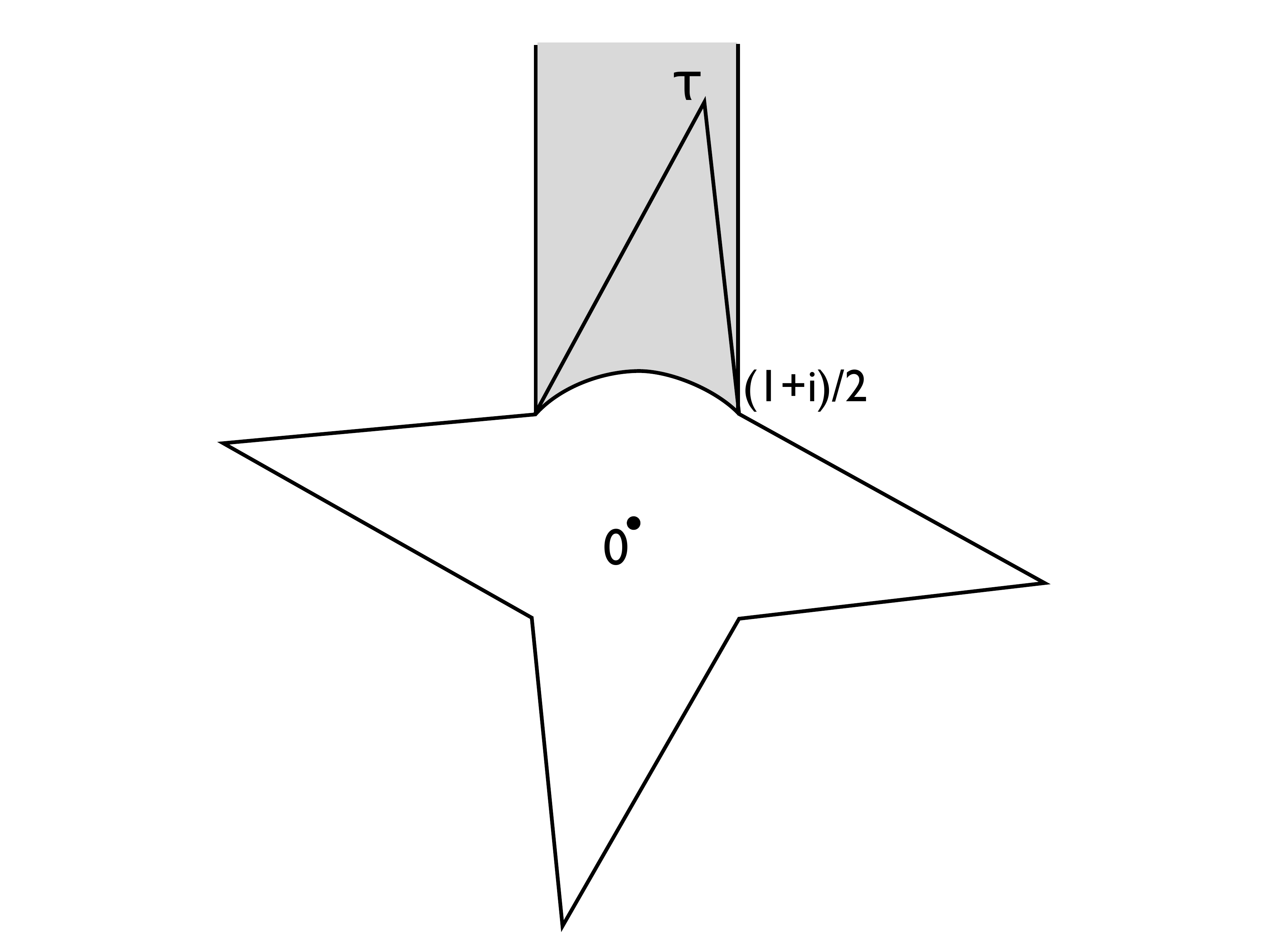}
  \end{center}
  \caption{\label{fig:pinwheels} {\sf For $\tau$ in the shaded domain $U$, the pinwheel $P_\tau$ is the polygon with vertices at $z=\frac{1\pm i}{2}$, $\frac{-1\pm i}{2}$, $\pm \tau$, and $\pm i\tau$.  Gluing together opposite sides on $P_\tau$ by translation gives a genus two surface admitting an action of $D_8$.  The one form $\omega_\tau$ induced by $dz$ is a $J$-eigenform and has a double zero.}}
\end{figure}

It is straightforward to identify the action of $D_8$ in most of the descriptions above.  The field $K_a$ has automorphisms $r(z,x)=(z,-x)$ and $J(z,x)=(iz/x^3,1/x)$.  Multiplication by the matrices $\spmat{ 1 & 0 \\ 0 & -1}$ and $\spmat{ 0 & 1 \\ -1 & 0 }$ preserve the polarized lattice $\Lambda_\tau$ giving automorphisms $r$ and $J$ of $A_\tau$.  The surface $X_\tau$ obtained from $P_\tau$ has an obvious order four automorphism $J_\tau$ and a genus two surface with an order four automorphism automatically admits a faithful action of $D_8$ (cf. Proposition \ref{prop:extendJ}).

The function relating the number $\tau$ determining the polygon $P_\tau$ and Abelian variety $A_\tau$ to the number $a$ determining the field $K_a$ is the modular function:
\[ a(\tau) = -2+\frac{1}{\lambda(\tau) \lambda(\tau+1)}. \]
The function $\lambda(\tau)$ is the function modular for $\Gamma(2) = \ker(\SL_2(\zz) \rightarrow \SL_2(\zz/2\zz))$ which covers the isomorphism $\lambda\co \hh/\Gamma(2) \xrightarrow{\sim} \cc\setminus \left\{ 0,1 \right\}$ sending the cusps $\Gamma(2) \cdot 0$, $\Gamma(2) \cdot 1$ and $\Gamma \cdot \infty$ to $0$, $1$ and $\infty$ respectively.  In Sections \ref{sec:d8family} and \ref{sec:hmsorbpts} we will prove:
\begin{thm} \label{thm:dsurfaces}
  Fix $\tau \in U$.  The surface $X_\tau$ obtained from the polygon $P_\tau$ admits a faithful $D_8$-action and satisfies:
  \[ \jac(X_\tau) \cong A_\tau \mbox{ and } \cc(X_\tau) \cong K_{a(\tau)}. \]
  The Jacobian $\jac(X_\tau)$ has complex multiplication if and only if $\tau$ is imaginary quadratic.
\end{thm}

\paragraph{Enumerating orbifold points on $W_D$.}  In addition to the formula in terms of class numbers for $e_2(W_D)$ appearing in Theorem \ref{thm:orbformula}, in Section \ref{sec:wdorbpts} we give a simple method for enumerating the orbifold points on $W_D$.  We define a finite set of {\em proper pinwheel prototypes} $E(D)$ consisting of triples of integers $(e,c,b)$ satisfying $D=-e^2+2bc$ along with certain additional conditions (cf. Equation \ref{eqn:prototypes}) and show:
\begin{thm}\label{thm:enumerateorbpts}
Fix a discriminant $D \geq 5$.  For any $(e,c,b) \in E(D)$, the surface:
\[ X_\tau \mbox{ with } \tau = (e+\sqrt{-D})/2c \]
is labeled by an orbifold point on $W_D$.  For discriminants $D>8$, the set $E(D)$ is in bijection with the points of orbifold order two on $W_D$.
\end{thm}
Since the field automorphisms of $\cc$ permute the set of $D_8$-surfaces with real multiplication by $\ord_D$ and the modular function $a \co Y_0(2) \rightarrow \cc$ is defined over $\qq$, the following is a corollary of Theorem \ref{thm:enumerateorbpts}:
\begin{thm}
  \label{thm:algebraicmodels}
  For $D \geq 5$, the polynomial:
  \[ f_D(t) = \prod_{(e,c,b) \in E(D)} \left( t-a\left( \frac{e+\sqrt{-D}}{2c} \right) \right)\left( t-a\left( \frac{-e+\sqrt{-D}}{2b}  \right) \right), \]
  has rational coefficients.  If $a$ is a root of $f_D(t)$, then the algebraic curve with $\cc(X) \cong K_a$ is labeled by an orbifold point on $W_D$.
\end{thm}
For example, when $D=76$ we have:
\[ E(76) = \left\{ (-2,2,20), (-2,4,10), (2,4,10) \right\}, \]
and the orbifold points on $W_{76}$ label the surfaces $X_{(-1+\sqrt{-19})/2}$, $X_{(-1+\sqrt{-19})/4}$ and $X_{(1+\sqrt{-19})/4}$.  Setting $q=e^{2\pi i \tau}$ and using the $q$-expansion:
  \[ a(\tau) = -2 - 256 q - 6144 q^2 - 76800 q^3 - 671744 q^4 + \dots\]
we can approximate the coefficients for $f_{76}(t)$ to high precision to show that:
\[ f_{76}(t) = t^3+3t^2+3459t+6913. \]
Table \ref{tab:minpoly} in Appendix \apporbptmodels\ lists the polynomials $f_D(t)$ for $D \leq 56$ computed by similar means.

\paragraph{Outline.} We conclude this Introduction with an outline of the proofs of our main results.
\begin{enumerate}
\item In \S\ref{sec:d8family}, we define and study the family $\M_2(D_8)$. The moduli space $\M_2(D_8)$ parametrizes pairs $(X,\rho)$ where $X \in \M_2$ and $\rho: D_8 \rightarrow \Aut(X)$ is injective.   There, our main goal is to prove the precise relationship between the surface $X_\tau$, the field $K_{a(\tau)}$ and the Abelian variety $A_\tau$ stated in Theorem \ref{thm:dsurfaces}.  We do so by showing, for $(X,\rho) \in \M_2(D_8)$, the quotient $X/\rho(r)$ has genus one, a distinguished base point and a distinguished point of order two (Proposition \ref{prop:d8involutionquotient}), giving rise to a holomorphic map $g\co \M_2(D_8) \rightarrow Y_0(2)$.  We then compute $\cc(X)$ and $\jac(X)$ in terms of $g(X,\rho)$ and show that the surface $X_\tau$ admits a $D_8$-action $\rho_\tau$ with $g(X_\tau,\rho_\tau)$ equal to the genus one surface $E_\tau = \cc/\zz\oplus \tau\zz$ with distinguished base point $Z_\tau = 0 +\zz\oplus\tau\zz$ and point of order two $T_\tau = 1/2+\zz\oplus \tau\zz$.
\item In \S\ref{sec:hmsorbpts}, we define and study the Hilbert modular surface $X_D$ and its orbifold points.  Orbifold points on $X_D$ correspond to Abelian varieties with automorphisms commuting with real multiplication.  It is well known that there are only a few possibilities for the automorphism group of a genus two surface (cf. Table \ref{tab:g2auts} in \S\ref{sec:d8family}), that the automorphism group of $X \in \M_2$ equals the automorphism group $\jac(X)$ and that every principally polarized Abelian surface is either a product of elliptic curves or a Jacobian.  Our classification of orbifold points on $\bigcup_D X_D$ in Theorem \ref{thm:hmsorbpts} is obtained by analyzing these possibilities.
\item In \S\ref{sec:wdorbpts}, we turn to the Weierstrass curve $W_D$.  Our classification of the orbifold points on $W_D$ in Corollary \ref{cor:wdorbpts} follows by analyzing which automorphism groups of genus two surfaces contain automorphisms which fix a Weierstrass point.  
\item We then prove the formula in Theorem \ref{thm:orbformula} by sorting the $D_8$-surfaces with complex multiplication by their orders for real multiplication commuting with $J$.  To do so, we embed the endomorphism ring of $\jac(X)$ in the rational endomorphism ring $\End(X/\rho(r) \times X/\rho(r)) \otimes \qq$, allowing us to relate the order for real multiplication on $\jac(X)$ to the order for complex multiplication on $X/\rho(r)$.
\item We conclude \S\ref{sec:wdorbpts} by giving a simple method for enumerating the $\tau \in U$ for which $X_\tau$ is labeled by an orbifold point on $W_D$.  For $\tau \in \qq(\sqrt{-D})$, we choose integers $e$, $k$, and $c$ so $\tau = (e+k\sqrt{-D})/(2c)$.  There is a rational endomorphism $T \in \End(\jac(X_\tau))\otimes \qq$ commuting the order four automorphism $J_\tau$ and generated real multiplication by $\ord_D$.  By writing down how $T$ acts on $H_1(X_\tau,\qq)$, we determine the conditions on $e$, $k$, and $c$ which ensure that $T$ preserves the lattice $H_1(X_\tau,\zz)$.
\item In \S\ref{sec:spin}, we sort the orbifold points on $W_D$ by spin component when $D \equiv 1 \bmod 8$.  For such discriminants, the orbifold points on $W_D$ are labeled by elements of the ideal class group $H(-4D)$.  We define a spin homomorphism:
\[ \epsilon_0 \co H(-4D) \rightarrow \zz/2\zz \]
which is the zero map if and only if $D$ is a perfect square.  We then relate the spin invariant of the orbifold point corresponding to the ideal class $[I]$ to the value $\epsilon_0([I])$ to give the formula in Theorem \ref{thm:spin}.
\item Finally, in \S\ref{sec:genus}, we collect the various formulas for topological invariants of $W_D$ and bound them to give bounds on the genus of $W_D$ and its components.
\end{enumerate}

\paragraph{Open problems.}
While the homeomorphism type of $W_D$ is now understood, describing the components of $W_D$ as Riemann surfaces remains a challenge.
\begin{pro}
  Describe $W_D$ as a hyperbolic orbifold and as an algebraic curve.
\end{pro}
Our analysis of the orbifold points on $W_D$ have given explicit descriptions of some complex multiplication points on $W_D$.  By the Andr\'e-Oort conjecture \cite{klingleryafaev:andreoort}, there are only finitely many complex multiplication points on $W_D$ and it would be interesting to find them.
\begin{pro}
  Describe the complex multiplication points on $W_D$.
\end{pro}
The complex multiplication points on $\M_2(D_8)$ lie on \teich\ curves and the complex multiplication points on $\M_2(D_{12})$ lie on complex geodesics in $\M_2$ with infinite fundamental group.  It would be interesting to find other examples of Shimura varieties whose complex multiplication points lie on interesting complex geodesics.
\begin{pro}
  Find other Shimura varieties whose complex multiplication lie on \teich\ curves.
\end{pro}
The divisors supported at cusps on modular curves generate a finite subgroup of the associated Jacobian \cite{manin:torsion}.  It would be interesting to know if the same is true for \teich\ curves.  The first Weierstrass curve with genus one is $W_{44}$.
\begin{pro}
  Compute the subgroup of $\jac(W_{44})$ generated by divisors supported at the cusps and points of order two.
\end{pro}
Algebraic geometers and number theorists have been interested in exhibiting explicit examples of algebraic curves whose Jacobians have endomorphisms.  A parallel goal is to exhibit Riemann surfaces whose Jacobians have endomorphisms as polygons in the plane glued together by translations as we did for the complex multiplication points on $\M_2(D_8)$ and $\M_2(D_{12})$.
\begin{pro}
  Exhibit surfaces whose Jacobians have complex multiplication as polygons in $\cc$ glued together by translation.
\end{pro}
There are very few Teichm\"uller curves $C \rightarrow \M_g$ whose images under the period mapping sending a surface to its Jacobian parametrize Shimura curves and they are classified in \cite{moeller:teichshimura}.  The families $\M_2(D_8)$ and $\M_2(D_{12})$ are examples of Teichm\"uller curves whose images under the period mapping are dense in Shimura curves.  Jacobians of $D_8$- and $D_{12}$-surfaces are dense in Shimura curves because they are characterized by their endomorphism ring.  The family $\M_2(D_8)$ (resp. $\M_2(D_{12})$) is a Teichm\"uller curves because the map $X \rightarrow X/\rho(J)$ (resp. $X \rightarrow X/\rho(Z)$) is branched over exactly four points.  It would be interesting to have a classification of such curves.
\begin{pro}
	Classify the Teichm\"uller curves whose images under the period mapping are dense in Shimura curves.
\end{pro}
\paragraph{Notes and references.}  For a survey of results related to the \teich\ geodesic flow, Teichm\"uller curves and relations to billiards see \cite{kms:ergodicity,masurtabachnikov:billiards,kontsevichzorich:components,zorich:flatsurfaces}.  Background about Abelian varieties, Hilbert modular surfaces and Shimura varieties can be found in \cite{vdgeer:hms}, \cite{shimura:cm}, \cite{birkenhakelange:cxabelianvarieties}, and \cite{shimura:endomorphisms}.  The orbifold points on $X_D$ are studied in \cite{prestel:elliptics} and the family $\M_2(D_8)$ has been studied in various settings, e.g. \cite{silhol:order4}.

\paragraph{Acknowledgments} The author would thank C. McMullen for many helpful conversations throughout this project as well as A. Preygel and V. Gadre for several useful conversations.  The author was partially supported in part by the National Science Foundation as a postdoctoral fellow.

\section{The $D_8$-family} \label{sec:d8family}
In this section we define and study the moduli space $\M_2(D_8)$ parametrizing pairs $(X,\rho)$ where $X \in \M_2$ and $\rho: D_8 \rightarrow \Aut(X)$ is injective.  In Section \ref{sec:introduction}, we defined a domain $U$, a polygon $P_\tau$ and a surface $X_\tau$ with order four automorphism $J_\tau$ for each $\tau \in U$, a field $K_a$ for $a \in \cc \setminus \left\{ \pm 2 \right\}$, an Abelian variety $A_\tau$ for $\tau \in \hh$ and a modular function $a : \hh \rightarrow \cc$.  Our main goal for this section is to prove the following proposition, establishing the claims in Theorem \ref{thm:dsurfaces} relating these different descriptions of $D_8$-surfaces:
\begin{prop}\label{prop:d8descriptions}
Fix $\tau \in U$.  There is an injective homomorphism $\rho: D_8 \rightarrow \Aut(X_\tau)$ with $\rho(J) = J_\tau$ and $X_\tau$ satisfies:
\[ \cc(X_\tau) \cong K_{a(\tau)} \mbox{ and } \jac(X_\tau) \cong A_\tau. \]
For any $(X,\rho) \in \M_2(D_8)$, there is a $\tau \in U$ so that there is an isomorphism $X \rightarrow X_\tau$ intertwining $\rho(J)$ and $J_\tau$.
\end{prop}
We prove Proposition \ref{prop:d8descriptions} by studying the quotients $E_\rho=X/\rho(r)$ as $(X,\rho)$ ranges in $\M_2(D_8)$.  We show that $E_\rho$ has genus one, a distinguished base point $Z_\rho$ and point of order two $T_\rho$ (cf. Proposition \ref{prop:d8involutionquotient}), allowing us to define a holomorphic map $g : \M_2(D_8) \rightarrow Y_0(2)$ by $g(X,\rho) = (E_\rho,Z_\rho,T_\rho)$.  We compute $\cc(X)$ and $\jac(X)$ in terms of $g(X,\rho)$ (Propositions \ref{prop:d8algebraicmodels} and \ref{prop:d8jacobians}) and then show that $X_\tau$ admits a $D_8$-action $\rho_\tau$ with $g(X_\tau,\rho_\tau) = (\cc/\zz\oplus \tau \zz,0+\zz\oplus \tau\zz,1/2+\zz\oplus \tau\zz)$ (Proposition \ref{prop:pinwheels}).

\paragraph{Surfaces with automorphisms.}  Let $G$ be a finite group.  We define the {\em moduli space of $G$-surfaces of genus $g$} to be the space:
\[ \M_g(G) = \left\{ (X,\rho) : X \in \M_g \mbox{ and } \rho\co G \rightarrow \Aut(X) \mbox{ is injective.} \right\} / \sim. \]
We will call two $G$-surfaces $(X_1,\rho_1)$ and $(X_2,\rho_2)$ equivalent and write $(X_1,\rho_1) \sim (X_2,\rho_2)$ if there is an isomorphism $f \co X_1 \rightarrow X_2$ satisfying $\rho_1(x) = f^{-1} \circ \rho_2(x) \circ f$ for each $x \in G$.  The set $\M_g(G)$ has a natural topology and a unique holomorphic structure so that the natural map $\M_g(G) \rightarrow \M_g$ is holomorphic.

\paragraph{Group homomorphisms and automorphisms.}  Any injective group homomorphism $h \co G_1 \rightarrow G_2$ gives rise to a holomorphic map $\M_g(G_2) \rightarrow \M_g(G_1)$.  In particular, the automorphism group $\Aut(G)$ acts on $\M_g(G)$. The inner automorphisms of $G$ fix every point on $\M_g(G)$ so the $\Aut(G)$-action factors through the outer automorphism group $\Out(G) = \Aut(G)/\operatorname{Inn}(G)$.  Note that $\Out(D_8)$ is isomorphic to $\zz/2\zz$ with the automorphism $\sigma(J) = J$ and $\sigma(r) = Jr$ representing the non-trivial outer automorphism.

\paragraph{Hyperelliptic involution and Weierstrass points.}  Now let $X$ be a genus two Riemann surface and $\Omega(X)$ be the space of holomorphic one forms on $X$.  The canonical map $X \rightarrow \mathbb P \Omega(X)^*$ is a degree two branched cover of the sphere branched over six points.  The {\em hyperelliptic involution} $\eta$ on $X$ is the Deck transformation of the canonical map and the Weierstrass points $X^W$ are the points fixed by  $\eta$.  Any holomorphic one form $\omega \in \Omega(X)$ has either two simple zeros at points $P$ and $Q \in X$ with $P =\eta(Q)$ or has a double zero at a point $P \in X^W$.

\paragraph{Automorphisms and permutations of $X^W$.}
Since it is canonically defined, $\eta$ is in the center of $\Aut(X)$ and any $\phi\in \Aut(X)$ induces an automorphism $\phi^\eta$ of the sphere $X/\eta$ and restricts to a permutation $\phi|_{X^W}$ of $X^W$.  The conjugacy classes in the permutation group $\Sym(X^W)$ are naturally labeled by partitions of six corresponding to orbit sizes in $X^W$, and we will write $[n_1,\dots,n_k]$ for the conjugacy class corresponding to the partition $n_1+\dots+n_k=6$.  We will denote the conjugacy class of $\phi|_{X^W}$ in the permutation of group of $X^W$ by $[\phi|_{X^W}]$.

In Table \ref{tab:g2auts}, we list the possibilities for $[\phi|_{X^W}]$ and, for each possibility, we determine the possibilities for the order of $\phi$, the number of points in $X$ fixed by $\phi$ and give the possible algebraic models for the pair $(X,\phi)$.  The claims are elementary to prove and well-known (cf. \cite{birkenhakelange:cxabelianvarieties} \S11.7 or \cite{bolza:sextics}).  Most can be proved by choosing an appropriately scaled coordinate $x \co X/\eta \rightarrow \chat$ so the action $\phi^\eta$ fixes $x^{-1}(0)$ and $x^{-1}(\infty)$.  From Table \ref{tab:g2auts}, we can show:
\begin{prop}
\label{prop:extendJ}
Suppose $X \in \M_2$ has an order four automorphism $\phi$.  The conjugacy class of $\phi|_{X^W}$ is $[1,1,2,2]$, the eigenforms for $\phi$ have double zeros, $\phi^2 = \eta$ and there is an injective group homomorphism $\rho \co D_8 \rightarrow \Aut(X)$ with $\rho(J) = \phi$.
\end{prop}
\begin{proof}
	From Table \ref{tab:g2auts}, we see that only $\phi$ with $[\phi|_{X^W}] = [1,1,2,2]$ have order four and that, for such automorphisms, there is a number $t \in \cc$ so that $\cc(X)$ is isomorphic to $\cc(x,y)$ with:
\[ y^2 = x (x^4 - t x^2 + 1), \]
and $\phi(x,y) = (-x,iy)$.  From this algebraic model we see that $\phi^2(x,y) = (x,-y)$ (i.e. $\phi^2 = \eta$), the eigenforms for $\phi$ are the forms $\frac{dx}{y} \cc^*$ and $x \frac{dx}{y} \cc^*$ which have double zeros, and that there is an injective group homomorphism $\rho\co D_8 \rightarrow \Aut(X)$ satisfying $\rho(J) = \phi$ and $\rho(r)(x,y) = (1/x,y/x^3)$.
\end{proof}

From Table \ref{tab:g2auts}, it also follows easily that:
\begin{prop}
	\label{prop:JZeigenforms}
	For any $(X,\rho) \in \M_2(D_8)$, the $\rho(J)$-eigenforms have double zeros.  For any $(X,\rho) \in \M_2(D_{12})$, the $\rho(Z)$-eigenforms have simple zeros.
\end{prop}
\begin{proof}
For $(X,\rho) \in \M_2(D_8)$, the automorphism $\rho(J)$ has order four so, by Proposition \ref{prop:extendJ}, has eigenforms with double zeros.  Also from Table \ref{tab:g2auts} we see that, for $(X,\rho) \in \M_2(D_{12})$, $[\rho(Z)|_{X^W}] = [3,3]$ so $\rho(Z)$ fixes no Weierstrass point and has eigenforms with simple zeros.
\end{proof}

\begin{table} \small
\begin{center}
\begin{tabular}{cccc}
\toprule
$[\phi|_{X^W}]$ & Algebraic model for $\left( X,\phi \right)$ & Order of $\phi$ & $\Fix(\phi)$ \\
\midrule
\parbox{1cm}{\centering $[1,1,1,$ \\$1,1,1]$} & \parbox{6.5cm}{ \centering $y^2 = x(x-1)(x-t_1)(x-t_2)(x-t_3)$ \\ and $\phi(x,y)=(x,y)$ or $(x,-y)$.} & 1 or 2 & $X$ or $X^W$ \\ \als
$[2,2,2]$ & \parbox{6.5cm}{\centering $y^2 = (x^2-1)(x^2-t_1)(x^2-t_2)$ \\ and $\phi(x,y)=(-x,y)$.} & 2 & $x^{-1}(0)$ \\ \als
$[1,1,2,2]$ & \parbox{6.5cm}{\centering $y^2 = x(x^4-t_1 x^2 +1)$  and\\ $\phi(x,y)=(-x,iy)$.} & 4  & $x^{-1}(\left\{ 0,\infty \right\})$\\ \als
$[1,1,4]$ & \parbox{6.5cm}{\centering $y^2 = x(x^4+1)$  and \\ $\phi(x,y)=(ix,(1+i)y/\sqrt{2})$.} & 8 & $x^{-1}(\left\{ 0,\infty \right\})$\\ \als
$[2,4]$ & \parbox{6.5cm}{\centering $y^2 = x(x^4+1)$ and \\ $\phi(x,y)=(i/x,(1+i)y/\sqrt{2} x^3)$.}& 8& $x^{-1}(\left\{ 0,\infty \right\})$\\ \als
$[3,3]$ & \parbox{6.5cm}{\centering $y^2 = (x^3-t_1^3)(x^3-t_1^{-3})$ and \\ $\phi(x,y)=(e^{2\pi i /3} x, y)$ or $(e^{2 \pi i /3} x, -y)$.}& 3 or 6 & $x^{-1}(\left\{ 0,\infty \right\})$\\ \als
$[1,5]$ & \parbox{6.5cm}{\centering $y^2 = (x^5+1)$ and  \\ $\phi(x,y)=(e^{2 \pi i/5}x,y)$ or $(e^{2\pi i/5}x,-y)$.}  & 5 or 10 & \parbox{1.5cm}{ \centering $x^{-1}(0)$ or\\ $x^{-1}(\infty)$ }\\ \als
$[6]$ & \parbox{6.5cm}{\centering $y^2 = x^6+1$  and  $\phi(x,y)=(e^{2 \pi i/6} x, y)$.} & 6 & $x^{-1}(0)$ \\ \als
\bottomrule
\end{tabular}
\caption{
\label{tab:g2auts} 
An automorphism $\phi$ of a genus two surface $X$ restricts to a bijection $\phi|_{X^W} \in \Sym(X^W)$.  When $\phi|_{X^W}$ is in one of the conjugacy classes above, the pair $(X,\phi)$ has an algebraic model $\cc(X) \cong \cc(x,y)$ with $x$, $y$ and $\phi$ satisfying the equations above for an appropriate choice of parameters $t_i \in \cc$. The omitted conjugacy classes do not occur as restrictions of automorphisms since the $\phi$ induces automorphism of the sphere $X/\eta$.}
\end{center}
\end{table}

\paragraph{Algebraic models for $D_8$-surfaces.}
In Section \ref{sec:introduction} we defined a field $K_a$ for each $a \in \cc \setminus \left\{ \pm 2 \right\}$ with an explicit action of $D_8$ on $K_a$ by field automorphisms.  Let $Y_a$ denote the genus two surface satisfying $\cc(Y_a) \cong K_a$ and let $\rho_a \co D_8 \rightarrow \Aut(Y_a)$ be the corresponding action of $D_8$.  We will eventually show that the map $f \co \cc \setminus \left\{ \pm 2 \right\} \rightarrow \M_2(D_8)$ given by $f(a)= (Y_a,\rho_a)$ is an isomorphism.  To start we will show $f$ is onto:
\begin{prop}  
\label{prop:fonto}
The map $a \mapsto (Y_a,\rho_a)$ defines a surjective holomorphic map:
\[ f \co\cc \setminus \left\{ \pm 2 \right\} \rightarrow \M_2(D_8). \]
In particular, $\dim_\cc(\M_2(D_8)) = 1$ and $\M_2(D_8)$ has one irreducible component.
\end{prop}
\begin{proof}
Fix $(X,\rho) \in \M_2(D_8)$.  From Table \ref{tab:g2auts} and Proposition \ref{prop:extendJ}, we have that $[\rho(r)|_{X^W}] = [2,2,2]$, $[\rho(J)|_{X^W}] = [1,1,2,2]$, $\rho(J)^2=\eta$ and $\Fix(r)$ is a single $\eta$-orbit consisting of two points.  These observations allow us to choose an isomorphism $x \co X/\eta \rightarrow \chat$ so that: (1) $\rho(r)(x) = -x$, (2) $x(\Fix(\rho(r))) = 0$ and (3) $x(P) = 1$ for some $P \in \Fix(\rho(J))$.  Since $\rho(J r J^{-1}) = \eta \rho(r)$, $J$ permutes the points in $X/\eta$ fixed by $\rho(r)^\eta$, i.e. $x^{-1}(\left\{ 0,\infty \right\})$, and must satisfy $\rho(J)(x)=1/x$.  If $Q$ is any point in $X^W$ not fixed by $\rho(J)$ and $t = x(Q)$, then:
\[ x(X^W) = \left\{ 1,-1,t,-t,1/t,-1/t \right\}. \]
The field  $\cc(X)$ is isomorphic to $\cc(x,y)$ where:
\[ y^2=(x^2-1)(x^4-ax^2+1) \mbox{ and } a = t^2+1/t^2. \]
Note that $a \not \in \cc \setminus\left\{ \pm 2 \right\}$ since the discriminant of $(x^2-1)(x^4-a x^2+1)$ is nonzer.  By conditions (1) and (2) on the coordinate $x$, we have $\rho(r)(x,y) = (-x,y)$. Since $\rho(J)(x)=1/x$, we have $\rho(J)(x,y) = (1/x,iy/x^3)$ or $(1/x,-iy/x^3)$.  In the first case, the obvious isomorphism between $X$ and $Y_a$ intertwines $\rho$ and $\rho_a$.  In the second case, the composition of $\rho(r)$ with the obvious isomorphism between $X$ and $Y_a$ intertwines $\rho$ and $\rho_a$.  In either case, $(X,\rho)$ is in the image of $f$.  
\end{proof}
\paragraph{Genus one surfaces with distinguished points of order two.}  
We now show that the quotient $(X/\rho(r))$ for any $D_8$-surface $(X,\rho)$ has genus one, a distinguished base point and point of order two.
\begin{prop}
\label{prop:d8involutionquotient}
For $(X,\rho) \in \M_2(D_8)$ we define:
\begin{equation} \label{eqn:d8quotient}
E_\rho = X/\rho(r), \mbox{ and } Z_\rho = \Fix(\eta \rho(r)) / \rho(r), T_\rho = \Fix(\rho(J))/\rho(r) \in E_\rho.
\end{equation}
The quotient $E_\rho$ has genus one and, under the group law on $E_\rho$ with identity element $Z_\rho$, the point $T_\rho$ is torsion of order two.
\end{prop}
\begin{proof}
By Proposition \ref{prop:fonto}, it is enough to establish the claims when $(X,\rho)$ is equivalent to $(Y_a,\rho_a)$.  From the equations defining $(Y_a,\rho_a)$ we see that $(E_{\rho_a},Z_{\rho_a},T_{\rho_a})$ satisfies:
\begin{equation}
\label{eqn:d8quotientmodel}
\begin{array}{c}
 \cc(E_{\rho_a}) = \cc(y,w=x^2) \mbox{ with } y^2 = (w-1)(w^2-aw+1), \\ 
Z_{\rho_a} = w^{-1}(\infty) \mbox{ and } T_{\rho_a} = w^{-1}(1).
\end{array}
\end{equation}
In particular, $E_{\rho_a}$ has genus one.  The function $(w-1)$ vanishes to order two at $T_{\rho_a}$ and has a pole of order two at $Z_{\rho_a}$, giving that $2 T_{\rho_a} = Z_{\rho_a}$ in the group law on $E_{\rho_a}$ with base point $Z_{\rho_a}$.
\end{proof}

Genus one Riemann surfaces with distinguished base point and point of order two are parametrized by the modular curve: 
\[ Y_0(2) = \left\{ (E,Z,T) : E \in \M_1 \mbox{ and } Z, P \in E \mbox{ satisfies } 2 \cdot P = 2 \cdot Z \right\} / \sim.\]
Two pairs $(E_1,Z_1,T_1)$ and $(E_2,Z_2,T_2)$ are equivalent if there is an isomorphism $E_1 \rightarrow E_2$ sending $Z_1$ to $Z_2$ and $T_1$ to $T_2$.  The modular curve can be presented as a complex orbifold as follows.  For $\tau \in \hh$, let 
\begin{equation}
E_\tau = \cc/\zz \oplus \tau\zz, Z_\tau = 0+\zz\oplus\tau\zz\mbox{ and } T_\tau = 1/2+\zz\oplus\tau\zz.
\end{equation}
The triples $(E_{\tau_1},Z_{\tau_1},T_{\tau_1})$ and $(E_{\tau_2},Z_{\tau_2},T_{\tau_2})$ are equivalent if and only if $\tau_1$ and $\tau_2$ are related by a M\"obius transformation in the group
\[\Gamma_0(2) =\left\{ \spmat{a & b \\ c & d} \in \SL_2(\zz) : c \equiv 0 \bmod 2 \right\}. \]
The map $\hh \rightarrow Y_0(2)$ given by $\tau \mapsto (E_\tau,Z_\tau,T_\tau)$ descends to a bijection on $\hh/\Gamma_0(2)$, presenting $Y_0(2)$ as a complex and hyperbolic orbifold.

A fundamental domain for $\Gamma_0(2)$ is the convex hull of $\left\{ 0,(-1+i)/2,(1+i)/2,\infty \right\}$ and $Y_0(2)$ is isomorphic to the $(2,\infty,\infty)$-orbifold.  The point of orbifold order two on $\hh/\Gamma_0(2)$ is $(1+i)/2 \cdot \Gamma_0(2)$ and corresponds to the square torus with the points fixed by an order four automorphism distinguished.
\begin{prop}
\label{prop:d8tox02extn}
Let $g \co \M_2(D_8) \rightarrow Y_0(2)$ be the map defined by:
\[ g(X,\rho) = (E_\rho,Z_\rho,T_\rho). \]
The composition $g \circ f \co \cc \setminus \left\{ \pm 2 \right\} \rightarrow Y_0(2)$ extends to a biholomorphism on $h: \cc \setminus \left\{ -2 \right\} \rightarrow Y_0(2)$ with $h(2) = (1+i)/2\cdot \Gamma_0(2)$.
\end{prop}
\begin{proof}
In Equation \ref{eqn:d8quotientmodel}, we gave an explicit model for $g(f(a))$.  By elementary algebraic geometry, as $a$ tends to $2$, $g(f(a))$ tends the square torus with the points fixed by an order four automorphism distinguished allowing us to extend the composition $g \circ f$ to a holomorphic map $h : \cc \setminus \left\{ -2 \right\} \rightarrow Y_0(2)$ with $h(2) = (1+i)/2 \cdot \Gamma_0(2)$.

The coarse spaces associated to $Y_0(2)$ and $\cc \setminus \left\{ -2 \right\}$ are both biholomorphic to $\cc^*$.  To show $h$ is a biholomorphism, it suffices to show $\deg(h)=1$.  To do so, consider the function $j \co Y_0(2) \rightarrow \cc$ sending $(E,Z,T)$ to the $j$-invariant of $E$.  The function $j$ has degree $[ \SL_2(\zz) : \Gamma_0(2)] = 3$.  From Equation \ref{eqn:d8quotientmodel}, it is straightforward to show:
\[ j\circ h(a) = 256(a+1)^3/(a+2). \]
Since $\deg(j \circ h) = \deg (j)$, we have $\deg(h) = 1$.   
\end{proof}
\begin{cor}
	\label{cor:fiso}
	The map $f \co \cc \setminus \left\{ \pm 2 \right\} \rightarrow \M_2(D_8)$ defined by $f(a) = (Y_a,\rho_a)$ is a biholomorphism.   The map $g\co \M_2(D_8) \rightarrow Y_0(2)$ is a biholomorphism onto its image, which is the complement of $(1+i)/2\cdot \Gamma_0(2)$.
\end{cor}
\begin{proof} Since $g \circ f$ extends to a biholomorphism and $f$ is onto, both $g$ and $f$ are biholomorphisms onto their images.  By Proposition \ref{prop:fonto}, $f$ is onto and, consequently, a biholomorphism.  This also implies that the image of $g$ is equal to the image of $g \circ f$, which we have shown is the complement of the point $(1+i)/2 \cdot \Gamma_0(2)$. \end{proof}

\paragraph{Outer automorphism.}
We have already seen that the non-trivial outer automorphism $\sigma$ of $D_8$ acts on $\M_2(D_8)$.  We will now identify the corresponding automorphisms of $\cc \setminus \left\{ \pm 2 \right\}$ and $Y_0(2)$ intertwining $\sigma$ with $f$ and $g$.

The modular curve $Y_0(2)$ has an {\em Atkin-Lehner involution} which we will denote $\sigma_X$.  The involution $\sigma_X$ is given by the formula $\sigma_X(E,Z_E,T_E) = (F,Z_F,T_F)$ where $F$ is isomorphic to the quotient $E/T_Z$ and $T_F$ is the point of order two in the image of the two torsion on $E$ under the degree two isogeny $E \rightarrow F = E/T_Z$.  In terms of our presentation of $Y_0(2)$ as a complex orbifold, $\sigma_X$ is defined by $\sigma_X(\tau \cdot \Gamma_0(2)) = \sigma_X(-1/2\tau \cdot \Gamma_0(2))$.  

The Riemann surface $\cc\setminus \left\{ \pm 2 \right\}$ parametrizing algebraic models for $D_8$-surfaces also has an involution:
\[ \sigma_A(a) = (-2a+12)/(a+2). \]
It is straightforward to show that $f$ and $g$ intertwine $\sigma$, $\sigma_A$ and $\sigma_X$, i.e.:
\[ \sigma \circ f = f \circ \sigma_A \mbox{ and } \sigma_X \circ g = g \circ \sigma. \]

\begin{prop}
\label{prop:extendequivalence}
Suppose $(X_1,\rho_1)$ and $(X_2,\rho_2)$ are $D_8$-surfaces, and there is an isomorphism $s : X_1 \rightarrow X_2$ intertwining $\rho_i(J)$.  Then $(X_1,\rho_1)$ is equivalent to either $(X_2,\rho_2)$ or $\sigma \cdot (X_2,\rho_2)$.
\end{prop}
\begin{proof}
We will show that $s$ intertwines $\rho_1(J^k r)$ with $\rho_2(r)$ for some $k$.  If $k=0$ or $k=2$, then either $s$ or $s \circ \rho_1(J)$ gives an equivalence between $(X_1,\rho_1)$ and $(X_2,\rho_2)$.  If $k=1$ or $k=3$, then either $s$ or $s \circ \rho_1(J)$ gives an equivalence between $(X_1,\rho_1)$ and $\sigma \cdot (X_2,\rho_2)$.

Setting $\phi = s \circ \rho_2(r) \circ s^{-1}$, our goal is to show that $\phi = \rho_1(J^k r)$ for some $k$.  For the $D_8$-surface $f(a)=(Y_a,\rho_a)$, the involution $\rho_a(r)$ interchanges the two points in $\Fix(\rho_a(J))$.  Since $(X_i,\rho_i)$ are equivalent to $(Y_{a_i},\rho_{a_i})$ for some $a_i$ by Proposition \ref{prop:fonto}, the same is true for $(X_i,\rho_i)$.  From $s \circ \rho_2(J) \circ s^{-1} = \rho_1(J)$ we see that $s^{-1}(\Fix(\rho_2(J)))=\Fix(\rho_1(J))$ and the composition $\phi \circ \rho_1(r)$ fixes both points in $\Fix(\rho_1(J))$.   This in turn implies that $\phi \circ \rho_1(r) = \rho_1(J)^k$ for some $k$.
\end{proof}

\paragraph{Fixed point of $\sigma$.}  The following proposition about the unique fixed point of $\sigma$ will be useful in our discussion of surfaces obtained from pinwheels.
\begin{prop}
Fix $(X,\rho) \in \M_2(D_8)$.  The following are equivalent:
\begin{enumerate}
\item There is an automorphism $\phi \in Aut(X)$ satisfying $\phi^2 = \rho(J)$,
\item $\sigma \cdot (X,\rho) \sim (X,\rho)$,
\item $(X,\rho)= f(6)$, and
\item $g(X,\rho)= \sqrt{-2}/2 \cdot \Gamma_0(2)$.
\end{enumerate}
\end{prop}
The proof is straightforward so we omit it.  

\paragraph{Cusps of $\M_2(D_8)/\sigma$.}  The coarse space associated to $\M_2(D_8)/\sigma \cong (\cc \setminus \left\{ \pm 2 \right\}) / \sigma_A$ is isomorphic to $\cc^*$ and has two cusps.  The following proposition gives a geometric characterization of the difference between these two cusps and will be important in our discussion of surfaces obtained from pinwheels.
\begin{prop}
For a sequence of $D_8$-surface $(X_i,\rho_i)$, the following are equivalent:
\begin{enumerate}
\item $X_i$ tend to a stable limit with geometric genus zero as $i \rightarrow \infty$,
\item $\cc(X_i) \cong K_{a_i}$ with $a_i \rightarrow \infty$, and
\item the quotients $X_i/\rho_{i}(r)$ diverge in $\M_1$ as $i \rightarrow \infty$.
\end{enumerate}
\end{prop}
The proof of this proposition is also straightforward so we omit it.

\paragraph{A modular function.}  In Section \ref{sec:introduction} we defined modular functions $\lambda(\tau)$ and $a(\tau)$.  We now show:
\begin{prop}
	\label{prop:d8algebraicmodels}
If $(X,\rho) \in \M_2(D_8)$ and $ \tau \in \hh$ satisfies $g(X,\rho) = (E_\tau,Z_\tau,T_\tau)$, then:
\[ \cc(X) \cong K_{a(\tau)} \mbox{ where }  a(\tau) = -2 + \frac{1}{\lambda(\tau)\lambda(\tau+1)}.\]
\end{prop}
\begin{proof}
Up to precomposition by $\tau \mapsto -1/2\tau$, the function $a : \hh \rightarrow \cc$ defined above is the unique holomorphic function satisfying (1) $a$ covers an isomorphism $\bar a :Y_0(2) \xrightarrow{\sim} \cc \setminus \left\{ -2 \right\}$ and (2) $a\left( (1+i)/2 \right) = 2$.  The biholomorphic extension $h : \cc \setminus \left\{ -2 \right\} \rightarrow Y_0(2)$ of $g \circ f$ from Proposition \ref{prop:d8tox02extn} has $h(2) = (1+i)/2 \cdot \Gamma_0(2)$, so $h^{-1}$ also satisfies (1) and (2).  As a consequence, $f(\bar a(E_\tau,Z_\tau,T_\tau))$ is equal to either $f(a(\tau))$ or $f(a(-1/2\tau))=\sigma \cdot f(a(\tau))$.  In either case, $\cc(X) \cong K_{a(\tau)}$.
\end{proof}
As we saw in the proof of Proposition \ref{prop:d8tox02extn}, the function $a(\tau)$ satisfies
\[ a(\tau)^3+3 a(\tau)^2+(3-j(\tau)/256) a(\tau)+1-j(\tau)/128 = 0 \]
where $j(\tau)$ is the function, modular for $\SL_2(\zz)$ and equal to the $j$-invariant of $E_\tau$.

\paragraph{Jacobians of surfaces with involutions.}  Our next goal is to compute the Jacobian of a $D_8$-surface $(X,\rho)$ in terms of $g(X,\rho) \in Y_0(2)$.  We start, more generally, by describing the Jacobian of a surface $X \in \M_2$ with an involution $\phi \in \Aut(X)$, $\phi \neq \eta$ in terms of the quotients $X/\phi$ and $X/\eta \phi$.

From Table \ref{tab:g2auts} we see that there are distinct complex numbers $t_1$ and $t_2$ so that:
\[ \cc(X) \cong \cc(x,y) \mbox{ with } y^2 = (x^2-1)(x^2-t_1)(x^2-t_2) \mbox{ and } \phi(x,y) = (-x,y). \]
The quotients $E = X/\phi$ and $F = X/\eta \phi$ have algebraic models given by
\[ 
\begin{array}{c}
\cc(E) \cong \cc(z_E=y,w_E=x^2) \mbox{ with } z_E^2 = (w_E-1)(w_E-t_1)(w_E-t_2) \mbox{ and } \\
\cc(F) \cong \cc(z_F=xy,w_F=x^2) \mbox{ with } z_F^2 = w_F (w_F-1)(w_F-t_1)(w_F-t_2). 
\end{array}
\]
The genus one surfaces $E$ and $F$ have natural base points $Z_E=\Fix(\eta \phi)/\phi = w_E^{-1}(\infty)$ and $Z_F=\Fix(\phi)/\eta \phi = w_F^{-1}(0)$ respectively.  Also, the image of $X^W$ under the map $X \rightarrow E \times F$ is the set:
\[ \Gamma^W = \left\{ (w_E^{-1}(t),w_F^{-1}(t)) \in E \times F : t = 1,t_1 \mbox{ or } t_2 \right\}. \]
The set $\Gamma^W$ generates a subgroup of order four of the two torsion on $E \times F$ under the group laws with identity element $Z_E \times Z_F$.
\begin{prop}
\label{prop:d4jacobians}
Suppose $X$, $\phi$, $E$, $F$ and $\Gamma^W$ are as above.  The Jacobian of $X$ satisfies:
\[ \jac(X) \cong E \times F / \Gamma^W \]
and the principal polarization on $\jac(X)$ pulls back to twice the product polarization on $E \times F$ under the quotient by $\Gamma^W$ map.
\end{prop}
\begin{proof}
Let $\psi : X \rightarrow E \times F$ be the obvious map.  By writing down explicit bases for $H_1(X,\zz)$, $H_1(E,\zz)$ and $H_1(F,\zz)$ using the algebraic models defined above, it is straightforward to check that the image of $\psi_* : H_1(X,\zz) \rightarrow H_1(E,\zz) \oplus H_1(F,\zz)$ has index four and that the symplectic form on $\psi_*(H_1(X,\zz))$ induced by the intersection pairing on $X$ extends to twice the ordinary symplectic form on $H_1(E,\zz) \oplus H_1(F,\zz)$.

The holomorphic map $\psi$ factors through a map on Jacobians:
\[ \psi: \jac(X) \rightarrow \jac(E) \times \jac(F) \]
whose degree is equal to $[ H_1(E,\zz) \oplus H_1(F,\zz) : \psi_*(H_1(X,\zz)) ] = 4$.  Under the identification of $\jac(X)$ with the Picard group $\Pic^0(X)$, the two torsion in $\jac(X)$ consists of degree zero divisors of the form $[P_i - P_j]$ with $P_i \neq P_j$ and $P_i, P_j \in X^W$.  The image of the two torsion in $\jac(X)$ under $\psi$ has image $\Gamma^W$ generating a subgroup of order four in $\jac(E) \times \jac(F)$.  The composition of $\psi$ with the quotient by $\Gamma^W$ map has degree 16, vanishes on the two-torsion and factors through multiplication by two on $\jac(X)$ to give an isomorphism.
\end{proof}

\paragraph{Jacobians of $D_8$-surfaces.}  With Proposition \ref{prop:d4jacobians}, it is now easy to compute $\jac(X)$ for a $D_8$-surface $(X,\rho)$ in terms of $g(X,\rho)$.  In Section \ref{sec:introduction}, we defined a principally polarized Abelian variety $A_\tau$ for each $\tau \in \hh$.
\begin{prop}\label{prop:d8jacobians}
If $(X,\rho) \in \M_2(D_8)$ and $ \tau \in \hh$ satisfies $g(X,\rho) = (E_\tau,Z_\tau,T_\tau)$, then $\jac(X) \cong A_\tau$.
\end{prop}
\begin{proof}
Set $\phi = \rho(r)$, $E = X /\phi$ and $F = X / \eta \phi$.  By our definition of $g$, $E$ is isomorphic to $E_\tau$.  Also, $\eta \phi = \rho(J) \circ \phi \circ \rho(J^{-1})$,  so $\rho(J)$ induces an isomorphism between $F$ and $E$.  In particular, $F$ is also isomorphic to $E_\tau$.  The image of $\Gamma^W \subset E \times F$ under the isomorphism to $E \times F \rightarrow E_\tau \times E_\tau$ is the graph of the induced action of $\rho(J)$ on the two torsion of $E_\tau$. Setting $T_\tau = 1/2 + \zz \oplus \tau \zz$, $Q_\tau = \tau/2 + \zz\oplus \tau \zz$ and $R_\tau = (\tau+1)/2 + \zz\oplus \tau \zz$, we have: 
\[ \Gamma^W = \left\{ (T_\tau,T_\tau),(Q_\tau,R_\tau),(R_\tau,Q_\tau) \right\}. \]
The Abelian variety $E_\tau \times E_\tau/\Gamma^W$, principally polarized by half the product polarization on $E_\tau \times E_\tau$, is easily checked to be equal to $A_\tau$.  (Note that the lattice $\Lambda_\tau$ used to define $A_\tau$ contains the vectors $\spmat{ 2 \\0}$, $\spmat{2 \tau \\ 0}$, $\spmat{ 0 \\ 2}$ and $\spmat{0 \\ 2 \tau}$).
\end{proof}

\paragraph{Pinwheels.}  In Section \ref{sec:introduction}, we defined a domain $U$ and associated to each $\tau \in U$ a polygonal pinwheel $P_\tau$ and a surface $X_\tau$ with order four automorphism $J_\tau$.  
\begin{prop}
	\label{prop:pinwheels}
Fix $\tau \in U$.  There is an injective homomorphism $\rho_\tau : D_8 \rightarrow \Aut(X_\tau)$ with $\rho_\tau(J) = J_\tau$ and $g(X_\tau,\rho_\tau) = (E_\tau,Z_\tau,P_\tau)$.
\end{prop}
\begin{proof}
By Proposition \ref{prop:extendJ}, the surface $X_\tau$ admits a faithful $D_8$-action $\rho_\tau : D_8 \rightarrow \Aut(D_8)$ with $\rho(J) = J_\tau$.  Since $U$ is simply connected, we can choose $\rho_\tau$ so $\tau \mapsto (X_\tau,\rho_\tau)$ gives a holomorphic map $p : U \rightarrow \M_2(D_8)$.

As depicted in Figure \ref{fig:d8cutandpaste}, the polygons $P_\tau$ and $P_{\tau+1}$ differ by a Euclidean cut-and-paste operation, giving an isomorphism between $X_\tau$ and $X_{\tau+1}$ intertwining $J_\tau$ and $J_{\tau+1}$.  
\begin{figure}
  \begin{center}
    \includegraphics[scale=0.35]{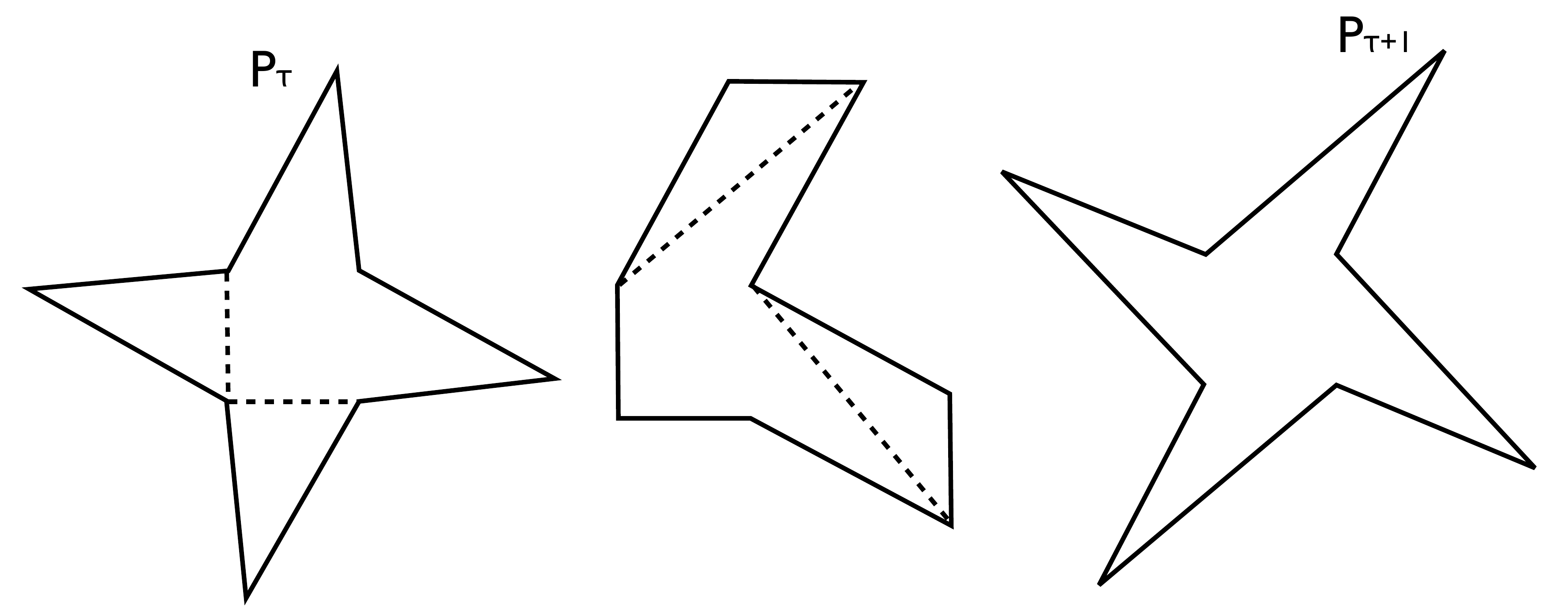}
\end{center}
\caption{\label{fig:d8cutandpaste} {\sf The surfaces $X_\tau$ and $X_{\tau+1}$ are isomorphic since the polygons $P_\tau$ and $P_{\tau+1}$ differ by a cut-and-paste operation.}
  } 
\end{figure}
Also, the polygons $P_\tau$ and $P_{-1/2\tau}$ differ by a Euclidean similarity, giving an isomorphism between $X_\tau$ and $X_{-1/2\tau}$ intertwining $J_\tau$ and $J_{-1/2\tau}$.  By Proposition \ref{prop:extendJ}, $p$ covers a holomorphic map $\bar p: U/\sim \rightarrow \M_2(D_8)/\sigma$ where $\tau \sim \tau+1$ and $\tau \sim -1/2\tau$.

The coarse spaces associated to both $U/\sim$ and $\M_2(D_8)/\sigma \cong (\cc \setminus \left\{ \pm 2 \right\}) / \sigma_A$ are biholomorphic to $\cc^*$.  Also, there is a unique point $\tau = \sqrt{-2}/2$ in $U$ for which $J_\tau$ is the square of an order eight automorphism, so $ p^{-1}(\Fix(\sigma)) = \left\{ \sqrt{-2}/2 \right\}$.  In particular $\bar p$ has degree one and is a biholomorphism.  There are precisely two biholomorphisms $U/\sim \rightarrow \M_2(D_8)/\sigma$ satisfying $\bar p(\sqrt{-2}/2) = \Fix(\sigma)/\sigma$.  They are distinguished by the geometric genus of the stable limit of $p(\tau)$ as $\Im(\tau)$ tends to $\infty$.  The geometric genus of $X_\tau$ as $\Im(\tau)$ tends to infinity is zero.

Another holomorphic map $p_1 : U \rightarrow \M_2(D_8)$ is given by $p_1(\tau) = g^{-1}(E_\tau,Z_\tau,T_\tau)$.  Since $p_1$ intertwines $\tau\mapsto -1/2\tau$ and the outer automorphism $\sigma$, and $(E_\tau,Z_\tau,T_\tau) = (E_{\tau+1},Z_{\tau+1},T_{\tau+1})$, $p_1$ also covers an isomorphism $\bar p_1: U/\sim \rightarrow \M_2(D_8)/\sigma$.  Moreover $p_1(\sqrt{-2}/2) = \Fix(\sigma)/\sigma$ and the limit of $p_1(\tau)$ diverges tends to the cusp of $\M_2(D_8)/\sigma$ with stable limit of genus zero.  So $\bar p_1 = \bar p$ and the proposition follows.
\end{proof}
Combining the results of this section we have:
\begin{proof}[Proof of Proposition \ref{prop:d8descriptions}]
Fix $\tau \in U$.  By Proposition \ref{prop:pinwheels}, the surface $X_\tau$ obtained from $P_\tau$ admits a $D_8$-action $\rho_\tau$ with $g(X_\tau,\rho_\tau) = (E_\tau,Z_\tau,P_\tau)$.  By Proposition \ref{prop:d8jacobians}, $\jac(X_\tau) \cong A_\tau$ and by Proposition \ref{prop:d8algebraicmodels} $\cc(X_\tau) \cong K_{a(\tau)}$.

Now fix any $(X,\rho) \in \M_2(D_8)$.  The domain $U$ is a fundamental domain for the group generated by $\Gamma_0(2)$ and the Atkin-Lehner involution $\tau \mapsto -1/2\tau$.  It follows that there is a $\tau \in U$ so $g(X,\rho) = (E_\tau,Z_\tau,T_\tau)$ or $g(\sigma \cdot(X,\rho)) = (E_\tau,Z_\tau,T_\tau)$.  In either case, by Proposition \ref{prop:pinwheels} and the fact that $g$ is an isomorphism onto its image, there is an isomorphism $X \rightarrow X_\tau$ intertwining $\rho(J)$ and $J_\tau$.
\end{proof}

\section{Orbifold points on Hilbert modular surfaces} \label{sec:hmsorbpts}
In this section, we discuss two dimensional Abelian varieties with real multiplication, Hilbert modular surfaces and their orbifold points.  Our main goals are to show: (1) the Abelian variety $A_\tau \cong \jac(X_\tau)$ has complex multiplication if and only if $\tau$ is imaginary quadratic, completing the proof of Theorem \ref{thm:dsurfaces} (Proposition \ref{prop:d8jaccm}) (2) the Jacobians of $D_8$- and $D_{12}$-surfaces with complex multiplication are labeled by orbifold points in $\bigcup_{D} X_D$ (Proposition \ref{prop:d8d12cmhmsorbpts}) and (3) establish the characterization of orbifold points on $\bigcup_D X_D$ of Theorem \ref{thm:hmsorbpts} (Propositions \ref{prop:hmsorbpts}).

\paragraph{Quadratic orders.}  Each integer $D \equiv 0$ or $1 \bmod 4$ determines a quadratic ring:
\[ \ord_D=\frac{\zz[t]}{(t^2-D t+D(D-1)/4)}.\]
The integer $D$ is called the {\em discriminant} of $\ord_D$.  We will write $\sqrt{D}$ for the element $2t-D \in \ord_D$ whose square is $D$ and define $K_D = \ord_D \otimes \qq$.  

We will typically reserve the letter $D$ for positive discriminants and the letter $C$ for negative discriminants.  For a positive discriminant $D>0$, the ring $\ord_D$ is totally real, i.e. every homomorphism of $\ord_D$ into $\cc$ factors through $\rr$.  For such $D$, we will denote by $\sigma_+$ and $\sigma_-$ the two homomorphisms $K_D \rightarrow \rr$ characterized by $\sigma_+(\sqrt{D}) > 0 > \sigma_-(\sqrt{D})$.  We will also use $\sigma_+$ and $\sigma_-$ for the corresponding homomorphisms $\SL_2(K_D) \rightarrow \SL_2(\rr)$.  The {\em inverse different} is the fractional ideal $\ord_D^\vee = \frac{1}{\sqrt{D}} \ord_D$ and is equal to the trace dual of $\ord_D$.

\paragraph{Unimodular modules.}  Now let $\Lambda_D = \ord_D \oplus \ord_D^\vee \subset K_D^2$.  The $\ord_D$-module $\Lambda_D$ has a unimodular symplectic form induced by trace:
\[ \left<(x_1,y_1),(x_2,y_2) \right> = \Tr^{K_D}_\qq(x_1 y_2 - x_2 y_1). \]
Up to symplectic isomorphism of $\ord_D$-modules, $\Lambda_D$ is the unique unimodular $\ord_D$-module isomorphic to $\zz^4$ as an Abelian group with the property that the action of $\ord_D$ is self-adjoint and proper (cf. \cite{mcmullen:dynamicsingenustwo}, Theorem 4.4).  Here, {\em self-adjoint} means $\left<\lambda v,w \right> = \left<v,\lambda w\right>$ for each $\lambda \in \ord_D$ and {\em proper} means that the $\ord_D$-module structure on $\Lambda_D$ is faithful and does not extend to a larger ring in $K_D$.

\paragraph{Symplectic $\ord_D$-module automorphisms.}  Let $\SL(\Lambda_D)$ denote the group of symplectic $\ord_D$-module automorphisms of $\Lambda_D$.  This group coincides with the $\ord_D$-module automorphisms of $\Lambda_D$ and equals:
\[\left\{ \begin{pmatrix} a & b\\c & d \end{pmatrix} : \begin{array}{c} ad-bc=1, a, d \in \ord_D, \\ b \in \sqrt{D} \ord_D \mbox{ and } c \in \ord_D^\vee \end{array}
\right\} \subset \SL_2(K_D) \]
with $A = \spmat{a & b \\ c & d}$ acting on $\Lambda_D$ by sending $(x,y)$ to $(ax+by,cx+dy)$.  The group $\SL(\Lambda_D)$ embeds in $\SL_2(\rr) \times \SL_2(\rr)$ via $A \mapsto (\sigma_+(A),\sigma_-(A))$ and acts on $\hh \times \hh$ by M\"obius transformations:
\[ (\tau_+,\tau_-)\cdot A = \left(\frac{\sigma_+(d \tau +b)}{\sigma_+(c \tau+a)},\frac{\sigma_-(d \tau+b)}{\sigma_-(c \tau+a)}\right) \]
where $\sigma_+(y \tau+x) = \sigma_+(y) \tau_+ + \sigma_+(x)$ and $\sigma_-(y \tau+x) = \sigma_-(y) \tau_- + \sigma_-(x)$.  

The following proposition characterizes the elements of $\SL(\Lambda_D)$ fixing every point in $\hh \times \hh$ and is elementary to verify:
\begin{prop} \label{prop:modkernel}
Let $h$ be the homomorphism $h\co \SL(\Lambda_D) \rightarrow \PSL_2(\rr) \times \PSL_2(\rr)$ given by $h(A) = (\pm \sigma_+(A),\pm \sigma_-(A))$.  For $A \in \SL(\Lambda_D)$, the following are equivalent:
\begin{itemize}
\item $A$ fixes every point in $\hh \times \hh$,
\item $A^2 = 1$, 
\item $A = \spmat{t & 0 \\ 0 & t}$ where $t \in \ord_D$ satisfies $t^2 =1$, and
\item $A$ is in $\ker(h)$.
\end{itemize}
The group $\ker(h)$ is isomorphic to the Klein-four group when $D = 1$ or $4$ and is cyclic of order two otherwise.
\end{prop}

\paragraph{Hilbert modular surfaces.}  The group $\PSL(\Lambda_D) = \SL(\Lambda_D)/\ker(h)$ acts faithfully and properly discontinuously on $\hh \times \hh$ and we define $X_D$ to be the quotient:
\[ X_D = \hh \times \hh / \PSL(\Lambda_D). \]
We will denote by $[\tau]$ the point in $X_D$ represented by $\tau \in \hh \times \hh$.  The complex orbifold is a typical example of a {\em Hilbert modular surface}.

\paragraph{Abelian varieties with real multiplication.}  Now let $B = \cc^2 / \Lambda$ be a principally polarized Abelian surface.  The {\em endomorphism ring} $\End(B)$ of $B$ is the ring of holomorphic homomorphisms from $B$ to itself.  We will say that $B$ admits {\em real multiplication by $\ord_D$} if there is a proper and self-adjoint homomorphism:
\[ \iota \co \ord_D \rightarrow \End(B). \]
Self-adjoint and proper here mean that $\iota$ turns the unimodular lattice $H_1(B,\zz) = \Lambda$ into self-adjoint and proper $\ord_D$-module.  

\paragraph{Moduli of Abelian varieties with real multiplication.}  Now let $A_2$ be the moduli space of principally polarized Abelian surfaces and set:
\[ A_2(\ord_D) = \left\{ (B,\iota) : 
\begin{array}{c}
B \in A_2 \mbox{ and } \iota : \ord_D \rightarrow \End(B) \\
 \mbox{ is proper and self-adjoint}
\end{array}
\right\} / \sim . \]
Here,  two pairs $(B_1,\iota_1)$ and $(B_2,\iota_2)$ are equivalent, and we write $(B_1,\iota_1) \sim (B_2,\iota_2)$, if there is a polarization preserving isomorphism $B_1 \rightarrow B_2$ intertwining $\iota_i$.

As we now describe and following \cite{mcmullen:dynamicsingenustwo} \S4 (see also \cite{birkenhakelange:cxabelianvarieties} Chapter 9), the Hilbert modular surface $X_D$ parametrizes $A_2(\ord_D)$ and presents $A_2(\ord_D)$ as a complex orbifold.  For $\tau =(\tau_+,\tau_-) \in \hh \times \hh$, define $\phi_\tau \co \Lambda_D \rightarrow \cc^2$ by:
\[ \phi_\tau(x,y)= (\sigma_+(x+y\tau),\sigma_-(x+y\tau)). \]
The image $\phi_\tau(\Lambda_D)$ is a lattice, and the complex torus $B_\tau = \cc^2/\phi_\tau(\Lambda_D)$ is principally polarized by the symplectic form on $\Lambda_D$.  For each $x \in \ord_D$, the matrix $\spmat{ \sigma_+(x) & 0 \\ 0 &  \sigma_-(x)}$ preserves the lattice $\phi_\tau(\Lambda_D)$ giving real multiplication by $\ord_D$ on $B_\tau$:
\[ \iota_\tau \co \ord_D \rightarrow \End(B_\tau). \]
For $A \in \SL(\Lambda_D)$, the embeddings $\phi_{\tau A }$ and $\phi_\tau$ are related by $\phi_{ \tau A} = C(A) \circ \phi_\tau \circ A$ where $C(A) = \spmat{ \sigma_+(a+c\tau) & 0 \\ 0 & \sigma_-(a+c\tau) }$.  From this we see that there is a polarization preserving isomorphism between $B_\tau$ and $B_{ \tau A}$ that intertwines $\iota_{\tau}$ and $\iota_{\tau A}$ and that the correspondence $\tau \rightarrow (B_\tau,\iota_\tau)$ descends to a map $X_D \rightarrow A_2(\ord_D)$.  This map is in fact a bijection and presents $A_2(\ord_D)$ as complex orbifold.

\paragraph{Complex multiplication.}  Now let $\ord$ be a degree two, totally imaginary extension of $\ord_D$.  We will say that $B \in A_2$ admits {\em complex multiplication by $\ord$} if there is a proper and Hermitian-adjoint homomorphism:
\[ \iota \co \ord \rightarrow \End(B). \]
Here, Hermitian-adjoint means that the symplectic dual of $\iota(x)$ acting on $H_1(B,\zz)$ is $\iota(\bar x)$ where $\bar x$ is the complex conjugate of $x$ and proper, as usual, means that $\iota$ does not extend to a larger ring in $\ord \otimes \qq$. 

For a one dimensional Abelian variety $E = \cc/\Lambda$ in $A_1$, we will say $E$ has {\em complex multiplication by $\ord_C$} if $\End(E)$ is isomorphic to $\ord_C$.  The {\em ideal class group} $H(C)$ is the set of invertible $\ord_C$-ideals modulo principal ideals and is well known to be in bijection with the set of $E \in A_1$ with $\End(A) \cong \ord_C$.  Since $\ord_C$ is quadratic, the invertible $\ord_C$-ideals coincide with the proper $\ord_C$-submodules of $\ord_C$.  The {\em class number} $h(C)$ is the order of the ideal class group $H(C)$.

We are now ready to determine which Jacobians of $D_8$-surfaces have complex multiplication.
\begin{prop}
\label{prop:d8jaccm}
Fix $\tau \in \hh$.  The Abelian variety $A_\tau$ has complex multiplication if and only if $\tau$ is imaginary quadratic.
\end{prop}
\begin{proof}
First suppose $\tau$ is imaginary quadratic.  The vector space $\Lambda_\tau \otimes \qq$ is stabilized by the matrices $\spmat{ 0 & 1 \\ -1 & 0 }$ and $\spmat{ \tau & 0 \\ 0 & \tau}$ which together generate a Hermitian adjoint embedding $\iota \co \qq(\tau,i) \rightarrow \End(A_\tau) \otimes \qq$.  The restriction of $\iota$ to the order $\ord = \iota^{-1}(\End(A_\tau))$ is complex multiplication by $\ord$ on $A_\tau$.

Now suppose $\tau$ is not imaginary quadratic.  We have seen that there is a degree four, surjective holomorphic map $f \co A_\tau \rightarrow E_\tau \times E_\tau$.  As is well known and is implied by the stronger Proposition \ref{prop:isogend}, such a map gives rise to an isomorphism between the rational endomorphism ring $\End(A_\tau) \otimes \qq$ and $\End(E_\tau \times E_\tau) \otimes \qq = M_2(\qq)$.  As a consequence, any commutative ring in $\End(A_\tau) \otimes \qq$ has rank at most two over $\qq$ and $A_\tau$ does not have complex multiplication by any order.
\end{proof}
We have now proved all the claims in our Theorem from \S \ref{sec:introduction} about $D_8$-surfaces:
\begin{proof}[Proof of Theorem \ref{thm:dsurfaces}]
The claims relating $X_\tau$, $K_{a(\tau)}$ and $A_\tau$ for $\tau \in U$ are established in Proposition \ref{prop:d8descriptions}.  The characterization of when $\jac(X_\tau)$ has complex multiplication is established in Proposition \ref{prop:d8jaccm}.
\end{proof}

\paragraph{Jacobians of $D_{12}$-surfaces with complex multiplication.}
In Appendix \ref{sec:d12family} we similarly define for $\tau \in \hh$ an Abelian variety $\wt A_\tau = \cc^2 / \wt \Lambda_\tau$ which, for most $\tau$, is the Jacobian of $D_{12}$-surface.  A nearly identical argument shows that, for $(X,\rho) \in \M_2(D_{12})$, the Jacobian $\jac(X)$ has complex multiplication if and only if $\jac(X)$ has complex multiplication by an order extending $\rho(Z)$, which happens if and only if $\jac(X) \cong \wt A_\tau$ with $\tau$ imaginary quadratic.

\paragraph{Orbifold points on Hilbert modular surfaces.}  We are now ready to study the orbifold points on $X_D$.  For $\tau \in \hh \times \hh$, we define the {\em orbifold order of $[\tau]$} in $X_D$ to be the order of the group $\Stab(\tau) \subset \PSL(\Lambda_D)$.  We will call $[\tau] \in X_D$ an {\em orbifold point} if the orbifold order of $[\tau]$ is greater than one.

The following proposition gives an initial characterization of the Abelian varieties labeled by such points:
\begin{prop}
	\label{prop:hmsorbptcharacterization}
	Fix $\tau \in \hh \times \hh$ and an integer $n > 2$.  The following are equivalent:
	\begin{enumerate}
		\item The point $\tau$ is fixed by an $A \in \SL_2(\Lambda_D)$ of order $n$.
		\item There is an automorphism $\phi \in \Aut(B_\tau)$ of order $n$ that commutes with $\iota_\tau(\ord_D)$.
		\item The homomorphism $\iota_\tau : \ord_D \rightarrow \End(B_\tau)$ extends to complex multiplication by an order containing $\ord_D[\zeta_n]$ where $\zeta_n$ is a primitive $n$th root of unity.
	\end{enumerate}
\end{prop}
\begin{proof}
	First suppose $(1)$ holds with $A = \spmat{a & b \\ c & d}$ and $\tau =  \tau A$.  We have seen that $\phi_\tau \circ A$ and $\phi_{\tau A} = \phi_\tau$ differ by multiplication by the matrix $C(A)=\spmat{ \sigma_+(a+c\tau) & 0 \\ 0 & \sigma_-(a+c\tau)}$.  It follows that $C(A)$ restricts to a symplectic automorphism of $\phi_\tau(\Lambda_D)$, giving rise to an automorphism $\phi \in B_\tau$ of order $n$ which commutes with $\iota_\tau$ since $A$ is $\ord_D$-linear.  Now suppose $(2)$ holds.  The homomorphism $\iota_\tau : \ord_D \rightarrow \End(B_\tau)$ extends to $\ord_D[\zeta_n]$ via $\iota_\tau(\zeta_n) = \phi$ and this extension is Hermitian-adjoint since $\phi$ is symplectic.  Finally, if $(3)$ holds, then the automorphism $\iota_\tau(\zeta_n)$ restricts to an $\ord_D$-module automorphism of $H_1(B_\tau,\zz) = \Lambda_D$, giving a matrix $A \in \SL_2(\ord_D \oplus \ord_D^\vee)$ of order $n$ and fixing $\tau$.
\end{proof}
We can now show that the $D_8$- and $D_{12}$-surfaces with complex multiplication give orbifold points on Hilbert modular surfaces:
\begin{prop}
\label{prop:d8d12cmhmsorbpts}
	The Jacobians of $D_8$- and $D_{12}$-surfaces with complex multiplication are labeled by orbifold points in $\bigcup_{D} X_D$.
\end{prop}
\begin{proof}
We will show that $D_8$ Jacobians with complex multiplication are labeled by orbifold points in $\bigcup_{D} X_D$.  A nearly identical argument shows the same is true of $D_{12}$ Jacobians with complex multiplication.  By Proposition \ref{prop:d8jaccm}, any such Jacobian is isomorphic to $A_\tau$ for some imaginary quadratic $\tau$.  The order $\ord$ constructed in the proof of Proposition \ref{prop:d8jaccm} contains $i$ since $\iota(i)$ is integral on $A_\tau$, and $(A_\tau,\iota|_{\ord \cap \rr})$ clearly satisfies condition (3) of Proposition \ref{prop:hmsorbptcharacterization}.
\end{proof}
We conclude this section by showing most of the orbifold points on $\bigcup_D X_D$ label Jacobians of $D_8$- and $D_{12}$-surfaces.  
\begin{prop} \label{prop:hmsorbpts}
	Fix an orbifold point $[\tau] \in X_D$.  At least one of the following holds:
	\begin{itemize}
		\item $B_\tau$ is a product of elliptic curves;
		\item $[\tau]$ is a point of orbifold order five on $X_5$;
		\item $B_\tau$ is the Jacobian of a $D_8$-surface with complex multiplication; or
		\item $B_\tau$ is the Jacobian of a $D_{12}$-surface with complex multiplication.
	\end{itemize}
\end{prop}
\begin{proof}
	By Proposition \ref{prop:hmsorbptcharacterization}, the Abelian variety $B_\tau$ labeled by $[\tau]$ has a symplectic automorphism $\phi$ of order greater than two and commuting with $\iota_\tau$.  It is well known that every principally polarized two dimensional Abelian variety is either a polarized product of elliptic curves of the Jacobian of a smooth genus two Riemann surface, and that the automorphism group of a genus two Riemann surface is isomorphic to the automorphism group of its Jacobian (cf. \cite{birkenhakelange:cxabelianvarieties}, Chapter 11).  If $B_\tau$ is not a product of elliptic curves, choose $X \in \M_2$ so $\jac(X)$ is isomorphic to $B_\tau$ and choose $\phi_0 \in \Aut(X)$ so that an isomorphism $\jac(X) \rightarrow B_\tau$ intertwines $\phi$ and $\phi_0$.  

	From Table \ref{tab:g2auts}, we see that $[\phi_0|_{X^W}]$ is in one of $[1,5]$, $[1,1,2,2]$, $[2,4]$, $[1,1,4]$, $[3,3]$ or $[6]$.  If $[\phi_0|_{X^W}]=[1,5]$, $X$ is the unique genus two surface with an order five automorphism, and $[\tau]$ is one of the points of order five on $X_5$.  If $[\phi_0|_{X^W}] = [1,1,2,2]$, $[2,4]$, or $[1,1,4]$, $B_\tau$ is the Jacobian of $D_8$-surface and has complex multiplication by Proposition \ref{prop:hmsorbptcharacterization}.  In the remaining cases, $B_\tau$ is the Jacobian of $D_{12}$-surface and, again, has complex multiplication by Proposition \ref{prop:hmsorbptcharacterization}.
\end{proof}

We have now proved the characterization of the orbifold points on $\bigcup_D X_D$ in Theorem \ref{thm:hmsorbpts}:
\begin{proof}[Proof of Theorem \ref{thm:hmsorbpts}]
By Proposition \ref{prop:d8d12cmhmsorbpts}, the Jacobians of $D_8$- and $D_{12}$-surfaces with complex multiplication give orbifold points on Hilbert modular surfaces.  By Proposition \ref{prop:hmsorbpts}, these points give all of the orbifold points on $X_D$ except those which are products of elliptic curves and the points of order five on $X_5$.
\end{proof}

\section{Orbifold points on Weierstrass curves} \label{sec:wdorbpts}
In this section we study the orbifold points on the Weierstrass curve $W_D$.  We start by recalling the definition of $W_D$.  An easy consequence of our classification of the orbifold points on $X_D$ in Section \ref{sec:hmsorbpts} gives the characterization of orbifold points on $\bigcup_D W_D$ of Corollary \ref{cor:wdorbpts}.  We then establish the formula in Theorem \ref{thm:orbformula} by sorting the $D_8$-surfaces with complex multiplication by the order for real multiplication commuting with $\rho(J)$.  To do so, we relate the order for real multiplication on $(X,\rho)$ to the order for complex multiplication on $X/\rho(r)$.  Finally, we conclude this section by giving a simple method for enumerating the $\tau \in U$ corresponding to orbifold points on $W_D$.

\paragraph{Eigenforms for real multiplication.}  For a principally polarized Abelian variety $B$ with real multiplication $\iota \co \ord_D \rightarrow \End(B)$, a place $\sigma_0 \co \ord_D \rightarrow \rr$ distinguishes a line of $\sigma_0$-eigenforms on $B$ satisfying $\iota(x)^* \omega = \sigma_0(x)\omega$ for each $x \in \ord_D$.  When $(B,\iota) = (B_\tau=\cc^2 / \Lambda_\tau,\iota_\tau)$, the $\sigma_+$ eigenforms are the multiples of $dz_1$ and the $\sigma_-$-eigenforms are the multiples of $dz_2$ where $z_i$ is the $i$th-coordinate on $\cc^2$.

For a Riemann surface $X \in \M_2$, the Abel-Jacobi map $X \rightarrow \jac(X)$ induces an isomorphism on the space of holomorphic one forms.  For Jacobians that admit real multiplication, a choice of real multiplication $\iota$ on $\jac(X)$ distinguishes $\sigma_+$- and $\sigma_-$-eigenforms on $X$.  Conversely, a one-form up to scale $[\omega]$ on $X$ that happens to be stabilized by real multiplication by $\ord_D$ on $\jac(X)$, there is a unique embedding $\iota^{[\omega]}_+ \co \ord_D \rightarrow \End(\jac(X))$ characterized by the requirement that $\iota^{[\omega]}_+(x) \omega = \sigma_+(x)\omega$ (cf. \cite{mcmullen:dynamicsingenustwo} \S 4).

\paragraph{The Weierstrass curve.}  The {\em Weierstrass curve} $W_D$ of discriminant $D$ is the moduli space:
\[ W_D = \left\{ (X,[\omega]) : \begin{array}{c} X \in \M_2  \mbox{ and $\omega$ is an eigenform for real }\\\mbox{ multiplication by $\ord_D$ with double zero }\end{array}\right\} / \sim. \]
Here, $[\omega]$ is a one-form up to scale, and $(X_1,[\omega_1])$ is equivalent to $(X_2,[\omega_2])$ and we write $(X_1,[\omega_1]) \sim (X_2,[\omega_2])$ if there is an isomorphism $\phi \co X_1 \rightarrow X_2$ with $\phi^* \omega_1 \in \cc^* \omega_2$.  The map $(X,[\omega]) \mapsto (\jac(X),\iota^{[\omega]}_+)$ embeds $W_D$ in the Hilbert modular surface $X_D$.  The natural immersion $W_D \rightarrow \M_2$ is shown to be a finite union of Teichm\"uller curves in \cite{mcmullen:billiards} (see also \cite{calta:periodicity}).

\paragraph{Orbifold points on Weierstrass curves.}  The Weierstrass curve can be presented as a complex orbifold in several equivalent ways.  One way is to use the $\SL_2(\rr)$-action on the moduli space of holomorphic one forms and Veech groups as in \cite{mcmullen:billiards}.  Another is to study the immersion $W_D \rightarrow X_D$ and give $W_D$ the structure of a suborbifold as in \cite{bainbridge:eulerchar}.  

The details of these presentations are not important for our discussion; instead, we simply define the notion of orbifold order and orbifold point on $W_D$.  For a Riemann surface $X$ and non-zero holomorphic one form $\omega \in \Omega(X)$, let $\Aut(X,\left\{ \pm \omega \right\})$ denote the subgroup of $\Aut(X)$ consisting of $\phi$ with $\phi^* \omega = \pm \omega$, let $\operatorname{SO}(X,\omega) = \Aut(X,[\omega])$ be the subgroup of $\Aut(X)$ consisting of automorphisms $\phi$ for which preserve $\omega$ up to scale and set:
\[ \PSO(X,\omega) = \operatorname{SO}(X,\omega)/\Aut(X,\left\{ \pm \omega \right\}). \]
The groups $\Aut(X,\left\{ \pm \omega \right\})$, $\operatorname{SO}(X,\omega)$ and $\PSO(X,\omega)$ only depend on $\omega$ up to scale.  

For a point $(X,[\omega]) \in W_D$, we define the {\em orbifold order of} $(X,[\omega])$ to be the order of the group $\PSO(X,\omega)$ and we will call $(X,[\omega])$ an {\em orbifold point} if its orbifold order is greater than one.  Using the characterization of orbifold points on $X_D$ in Proposition \ref{prop:hmsorbptcharacterization} it is straightforward to check that $(X,[\omega]) \in W_D$ is an orbifold point if and only if the pair $(\jac(X),\iota^{[\omega]}_+)$ is an orbifold point on $X_D$.

\paragraph{Orbifold points on $W_D$ and $D_8$-surfaces with complex multiplication.}
Recall from Section \ref{sec:introduction} that we defined a domain $U$ and associated to each $\tau \in U$ a polygonal pinwheel $P_\tau$ so the surface $X_\tau = P_\tau/\sim$ admits a faithful action of $D_8$.  Let $J_\tau$, as usual, denote the obvious order four automorphism of $X_\tau$ obtained by counterclockwise rotation of $P_\tau$ and let $\omega_\tau$ denote the eigenform for $J_\tau$ obtained from $dz$ on $P_\tau$.
\begin{prop}
\label{prop:d8orborders}
For $\tau \in U$, the group $\PSO(X_\tau,\omega_\tau)$ is cyclic of order two except when $\tau = \sqrt{-2}/2$, in which case $\PSO(X_\tau,\omega_\tau)$ is cyclic of order four.
\end{prop}
\begin{proof}
Recall from the proof of Proposition \ref{prop:pinwheels} that $J_\tau$ extends to a faithful action $\rho \co D_8 \rightarrow \Aut(X)$ with $\rho(J)=J_\tau$.  From the fact that $\cc(X_\tau) \cong K_{a(\tau)}$, it is easy to verify that $\rho$ is an isomorphism except when $(X,\rho)$ is the fixed point of the outer automorphism $\sigma$, i.e. $\tau = \sqrt{-2}/2$.  For $\tau \neq \sqrt{-2}/2$, the group $\operatorname{SO}(X_\tau,\omega_\tau)$ is generated by $J_\tau$, $\Aut(X_\tau,\left\{ \pm \omega_\tau \right\})$ is generated by the hyperelliptic involution $J_\tau^2$ and $\PSO(X,\omega_\tau)$ is cyclic of order two.  When $\tau = \sqrt{-2}/2$, $P_\tau$ is the regular octagon and it is easy to verify that $\PSO(X_\tau,\omega_\tau)$ is cyclic of order four.
\end{proof}
We now show that the $D_8$-surfaces with complex multiplication give orbifold points on Weierstrass curves:
\begin{prop}
\label{prop:d8cmareorbpts}
If $\tau \in U$ is imaginary quadratic, then there is a discriminant $D>0$ so the surface with one form up to scale $(X_\tau,[\omega_\tau])$ is an orbifold point on $W_D$.
\end{prop}
\begin{proof}
Fix an imaginary quadratic $\tau \in U$.  As we saw in the proof of Proposition \ref{prop:d8jaccm}, the order four automorphism $J_\tau$ on $X_\tau$, extends to complex multiplication by an order $\ord$ in $\qq(\tau,i)$.  If $D$ is the discriminant of $\ord \cap \rr$, we see that $\omega_\tau$ is an eigenform for real multiplication by $\ord_D$.  Also, $\omega_\tau$ has a double zero, as can be seen directly by counting cone-angle around vertices of $P_\tau$ or can be deduced from the fact that $\omega_\tau$ is stabilized by the order four automorphism $J_\tau$ (cf. Proposition \ref{prop:extendJ}).  From this we see that $(X_\tau,[\omega_\tau]) \in W_D$ and is an orbifold point on $W_D$ since $\PSO(X_\tau,\omega_\tau)$ has order at least two.
\end{proof}
We can now show that almost all of the orbifold points on $\bigcup_D W_D$ are $D_8$-surfaces with complex multiplication:
\begin{prop}
\label{prop:wdorbpts}
Suppose $(X,[\omega]) \in W_D$ is an orbifold point.  One of the following holds:
\begin{itemize}
\item $(X,[\omega])$ is the point of orbifold order five on $W_5$, 
\item $(X,[\omega])$ is the point of orbifold order four on $W_8$, or
\item $(X,[\omega]) = (X_\tau,[\omega_\tau])$ for some imaginary quadratic $\tau \in U$ with $\tau \neq \sqrt{-2}/2$.
\end{itemize}
In particular, for $D> 8$, all of the orbifold points on $W_D$ have orbifold order two.
\end{prop}
\begin{proof}
By our definition of orbifold point, $X$ has an automorphism $\phi$ stabilizing $\omega$ up to scale and with $\phi^* \omega$ not equal to $\pm \omega$.  Such a $\phi$ must fix the zero of $\omega$ which is a Weierstrass point and the conjugacy class $[\phi|_{X^W}]$ is one of $[1,5]$, $[1,1,4]$ or $[1,1,2,2]$.  In the first case, $(X,[\omega])$ is the point of order five on $W_5$.  In the remaining cases, $(X,[\omega]) = (X_\tau,[\omega_\tau])$ for some imaginary quadratic $\tau \in U$.  When $\tau=\sqrt{-2}/2$ and $(X,[\omega])$ is the point of orbifold order four on $W_8$ obtained from the regular octagon and when $\tau \neq \sqrt{-2}/2$, $(X,[\omega])$ is a point of orbifold order two for some larger discriminant.
\end{proof}
We have now proved the claims in Corollary \ref{cor:wdorbpts}:
\begin{proof}[Proof of Corollary \ref{cor:wdorbpts}]
By Proposition \ref{prop:d8cmareorbpts}, the $D_8$-surfaces with complex multiplication give orbifold points on $\bigcup_D W_D$.  By Proposition \ref{prop:wdorbpts}, they give all of the orbifold points except for the point of order five on $W_5$.
\end{proof}

\paragraph{Isogeny and endomorphism.}  In light of Proposition \ref{prop:wdorbpts}, to give a formula for the number $e_2(W_D)$ of points of orbifold order two on $W_D$, we need to sort the $D_8$-surfaces $(X,\rho)$ with complex multiplication by order for real multiplication commuting with $\rho(J)$.  To do so, we relate this order to the order for complex multiplication on $E_\rho = X/\rho(r)$.

We saw in Section \ref{sec:d8family} that there is an isogeny $\jac(X) \rightarrow E_\rho \times E_\rho$ such that the polarization on $\jac(X)$ is twice the pullback of the product polarization on $E_\rho \times E_\rho$.  The following proposition will allow us to embed the endomorphism ring of $\jac(X)$ in the rational endomorphism ring $E_\rho \times E_\rho$:
\begin{prop}
  \label{prop:isogend}
  Suppose $f \co A \rightarrow B$ is an isogeny between principally polarized Abelian varieties with the property the polarization on $A$ is $n$-times the polarization pulled back from $B$.  The ring $\End(B)$ is isomorphic as an involutive algebra to the subring $R_f \subset \End(A) \otimes \qq$ given by:
  \[ R_f = \frac{1}{n} \left\{ \phi \in \End(A) : \phi \left( \ker (n f) \right) \subset \ker(f) \right\}. \]
\end{prop}
\begin{proof}
Let $f^*\co B \rightarrow A$ denote the isogeny dual to $f$.  The condition on the polarization implies that $f \circ f^*$ and $f^* \circ f$ are the multiplication by $n$-maps on $B$ and $A$ respectively.  We will show that the map: 
  \begin{align*}
    \psi \co \End(B) &\rightarrow \End(A) \otimes \qq \\
    \phi &\mapsto \frac{1}{n} f^* \circ \phi \circ f
  \end{align*}
  is an injective homomorphism and has image $\psi(\End(B)) = R_f$.    The map $\psi$ is easily checked to be a homomorphism and is injective since rationally it is an isomorphism, with inverse is given by $\psi^{-1}(\phi) = \frac{1}{n} f \circ \phi \circ f^*$.

  For an integer $k > 0$, let $B[k]$ denote the $k$-torsion on $B$.  The image of $\End(B)$ is contained in $R_f$ since:
  \[ \ker(n f) \xrightarrow{f} B[n] \xrightarrow{\phi} B[n] \xrightarrow{f^*} \ker(f) \]
  for any $\phi \in \End(B)$.  To see that the image of $\psi$ is all of $R_f$, fix $\phi_0 \in \End(A) \otimes \qq$ satisfying $\phi_0(\ker(nf)) \subset \ker(f)$.  The endomorphism $\phi = \frac{1}{n^2} f \circ \phi_0 \circ f^*$ is integral on $B$ since:
  \[ B[n^2] \xrightarrow{f^*} \ker(nf) \xrightarrow{\phi} \ker(f) \xrightarrow{f} 0. \]
  Since $\psi(\phi) = \phi_0$, the image of $\psi$ is all of $R_f$.  
  
  The isomorphism between $\End(B)$ and $R_f$ has $\psi(\phi)^* = \psi(\phi^*)$ since $(f^* \phi f)^* = f^* \phi^* f$.
\end{proof}

\paragraph{Invertible modules over finite rings.}   From Proposition \ref{prop:isogend} and the fact that $\jac(X)$ admits a degree four isogeny to $E_\rho \times E_\rho$ with $E_\rho=X/\rho(r)$, we see that the discriminant of the order for real multiplication on $\jac(X)$ commuting with $\rho(J)$ is, up to factors of two, equal to $-C$ where $C$ is the discriminant of $\End(E_\rho)$.  To determine the actual order for real multiplication, we first need to determine the possibilities for $E_\rho[4]$ as an $\ord_C$-module.  The following proposition is well known:
\begin{prop}
  \label{prop:quadpic}
  Let $\order$ be an imaginary quadratic order and $I$ be an $\ord$-ideal which is a proper $\ord$-module.  The module $I/nI$ is isomorphic to $\ord / n \ord$ as $\ord$-modules.
\end{prop}
For maximal orders, see Proposition 1.4 in \cite{silverman:ellipticcurves}.  For non-maximal orders, Proposition \ref{prop:quadpic} follows from the fact that $\ord$-ideals which are proper $\ord$-modules are invertible (\cite{lang:ellipticfunctions}, \S 8.1) and standard commutative algebra.

\paragraph{Formula for number and type of orbifold points on $W_D$.}
We are now ready to prove our main theorem giving a formula for the number and type of orbifold points on $W_D$:
\begin{proof}[Proof of Theorem \ref{thm:orbformula}]
We have already seen that, for discriminants $D>8$, all of the orbifold points on $W_D$ have orbifold order two.  It remains to establish the formula for the number of such points.

Recall from Section \ref{sec:d8family} that we constructed a holomorphic map $g: \M_2(D_8) \rightarrow Y_0(2)$ given by $g(X,\rho) = (E_\rho,Z_\rho,T_\rho)$ which is an isomorphism onto the complement of the point of orbifold order two $(1+i)/2 \cdot \Gamma_0(2)$.  For each $E \in \M_1$ without automorphisms other than the elliptic involution, there are exactly three $D_8$-surfaces $(X,\rho)$ with $X/\rho(r)$ isomorphic to $E$.  

Now fix an imaginary quadratic discriminant $C< 0$.  Setting $E = \cc / \ord_C$, $Z = 0 + \ord_C$, $T_1 = 1/2+\ord_C$, $T_2 = (C+\sqrt{C})/4 + \ord_C$ and $T_3 = T_1+T_2$, the three $D_8$-surfaces covering $E$ are the three surfaces $(X_i,\rho_i) = g^{-1}(E,Z,T_i)$.  Using Proposition \ref{prop:isogend}, it is straightforward to calculate the order for real multiplication on $\jac(X_i)$ commuting with $\rho_i(J)$ and we do so in Table \ref{tab:orders}.
\begin{table}
\[
\begin{array}{|c|c|c|c|}
  \hline 
  C \bmod 16& D_1 & D_2 & D_3 \\ \hline
  \rule{0pt}{15pt}
  0 & -4C & -C & -4C \\
  4 & -4C & -4C & -C \\
  8 & -4C & -C & -4C \\
  12 & -4C & -4C & -C/4 \\
  1 \mbox{ or } 9 & -4C & -4C & -4C \\
  5 \mbox{ or } 13 & -4C & -4C & -4C \\
  \hline
\end{array}
\]
\caption{ \label{tab:orders} {\sf The elliptic curve $E = \cc/\ord_C$ is covered by three $D_8$-surfaces $(X_i,\rho_i)$.  The discriminant $D_i$ of the order for real multiplication on $\jac(X_i)$ commuting with $J$ is computed using Proposition \ref{prop:isogend}.}}
\end{table}

For an arbitrary $E$ with $\End(E) \cong \ord_C$, we have that $E[4] \cong \ord_C/4\ord_C$ as $\ord_C$-modules by Proposition \ref{prop:quadpic} and the orders for real multiplication commuting with $\rho(J)$ on the $D_8$-surfaces covering $E$ are the same as the orders when $E = \cc / \ord_C$.  The formula for $e_2(W_D)$ follows easily from this and the fact that there are precisely $h(\ord_C)$ genus one surfaces with $\End(E) \cong \ord_C$.  The small discriminants where the $D_8$-surfaces labeled by orbifold points on $W_D$ cover genus one surfaces with automorphisms and handled by the restriction $D>8$ and by replacing $h(\ord_C)$ with the reduced class number $\wt h (\ord_C)$.  Note that the factor of two in the formula for $e_2(W_D)$ comes from the fact that the $D_8$-surfaces $(X,\rho)$ and $\sigma \cdot (X,\rho)$ have isomorphic $J$-eigenforms and therefore label the same orbifold point on $W_D$.
\end{proof}

\paragraph{Enumerating orbifold points on $W_D$.}  We conclude this section by giving a simple method for enumerating the $\tau \in U$ for which $(X_\tau,[\omega_\tau]) \in W_D$.

Fix a discriminant $D>0$ and a $\tau \in \qq(\sqrt{-D})$.  We can choose integers $e$, $k$, $b$ and $c$ so $k^2 D = -e^2+2bc$ and $\tau = (e+k\sqrt{-D})/(2c)$.  The vectors $v_1=\spmat{1\\1}$, $v_2 = \spmat{1 \\-1}$, $v_3=\spmat{\tau\\ \tau+1}$ and $v_4=\spmat{\tau \\ -\tau-1}$ generate the lattice $\Lambda_\tau$.  The lattice $\Lambda_\tau$ is preserved by multiplication by $\spmat{0 & 1 \\ -1 & 0}$ giving an automorphism $\phi \in \Aut(A_\tau)$, and the vector space $\Lambda_\tau \otimes \qq$ is preserved by the matrix $\spmat{ \sqrt{-D} & 0 \\ 0 & \sqrt{-D}}$ giving a rational endomorphism $T \in \End(A_\tau \otimes \qq)$.  Together, $\phi$ and $T$ generate a Hermitian-adjoint homomorphism $\iota : \qq(\sqrt{-D},i) \rightarrow \End(A_\tau) \otimes \qq$.  

We want to give conditions on $e$, $k$, $b$ and $c$ so that the quadratic ring of discriminant $D$ in $\qq(\sqrt{-D},i)$ acts integrally on $\Lambda_\tau$.  This ring is generated by $S=(D+\phi T)/2$ and it is straightforward to check that, in the basis $v_1,\dots,v_4$ for $\Lambda_\tau \otimes \qq$, $S$ acts by multiplication by the matrix:
\[S = \frac{1}{2k} \begin{pmatrix}
    Dk+c & -e-c & 0 & 2c \\
    c+e & Dk-c & -2c & 0\\
    0 & -e-b & Dk+c & c+e \\
    b+e & 0 & -e-c & Dk-c
  \end{pmatrix}.
\]
From this we see that $S$ is integral, and $\jac(X_\tau) = \cc^2 /\Lambda_\tau$ has real multiplication by an order containing $\ord_D$ and commuting with $\phi$, if and only if $k$ divides $c$, $b$ and $e$ and $D \equiv e/k \equiv b/k \equiv c/k \bmod 2$.

\paragraph{Pinwheel prototypes.} Motivated by this, define the set of {\em pinwheel prototypes of discriminant $D$}, which we denote by $E_0(D)$, to be the collection of triples $(e,c,b) \in \zz^3$ satisfying:
\begin{equation} 
\label{eqn:prototypes}
\begin{array}{c} 
D = -e^2+2bc \mbox{ with } D \equiv e \equiv c \equiv b \bmod 2 \mbox{ and } \abs{e} \leq c \leq b \\
\mbox{ and if } \abs{e} = c \mbox{ or } b = c \mbox{ then } e \leq 0.
\end{array}
\end{equation}
Note that if $(e,c,b) \in E_0(D)$, then $(fe,fc,fb) \in E_0(f^2D)$.  We define $E(D)$ to be the set of {\em proper pinwheel prototypes} in $E_0(D)$, i.e. those which do not arise from smaller discriminants in this way.  Finally define $\tau(e,c,b) = (e+\sqrt{-D})/2c$.  

We can now prove the following proposition, which is a more precise version of Theorem \ref{thm:enumerateorbpts}:
\begin{prop}
Fix a discriminant $D$ and $(e,c,b) \in E(D)$.  The one form up to scale $(X_\tau,[\omega_\tau])$ where $\tau=(e+\sqrt{D})/(2c)$ is an orbifold point on $W_D$ and, for $D> 8$, the set $E(D)$ is in bijection with the orbifold points on $W_D$:
\[ e_2(W_D) = \# E(D). \]
\end{prop}
\begin{proof}
By our discussion above, the first two conditions on pinwheel prototypes ($D = -e^2+2bc$ and $D \equiv e \equiv c \equiv b \bmod 2$) are equivalent to the requirement that $(X_\tau,[\omega_\tau])$ where $\tau = \tau(e,c,b)$ is an eigenform for real multiplication by an order containing $\ord_D$.  The condition that $(e,c,b)$ is proper ensures that the order for real multiplication with eigenform $\omega_\tau$ is $\ord_D$.  The condition that $\abs{e} \leq c \leq b$ is equivalent to the condition that $\tau(e,c,b)$ is in the domain $U$.  The condition that $e \leq 0$ when $\abs{e} = c$ or $b=c$ ensures that $\Re(\tau(e,c,b)) \leq 0$ whenever $\tau(e,c,b)$ is in the boundary of $U$.  
\end{proof}

\paragraph{Bounds on $e_2(W_D)$.} It is easy to see that the conditions defining pinwheel prototypes ensure $c \leq \sqrt{D}$, from which it is easy to enumerate the prototypes in $E(D)$ and prove:
\begin{prop}\label{prop:orbptbounds}
For discriminants $D>8$, the number of points of orbifold order two satisfies:
\[ e_2(W_D) \leq D/2. \]
\end{prop}
\begin{proof}
The integers $e$ and $c$ determine $b$ since $b = (D+e^2)/(2c)$.  From $\abs{e} \leq c \leq \sqrt{D}$ and the fact that $e$ and $c$ are congruent to $D \bmod 2$, we have that $e$ ranges over at most $\sqrt{D}$ possibilities and $c$ ranges over at most $\sqrt{D}/2$ possibilities so $e_2(W_D) < D/2$.
\end{proof}
\paragraph{Examples.} The sets $E(D)$ for some small discriminants $D$ are:
\[ \begin{array}{c} 
E(5) = \left\{ (-1,1,3) \right\}, E(8) = \left\{ (0,2,2) \right\}, E(9) = \left\{ (-1,1,5) \right\}, E(12) = \left\{ (-2,2,4) \right\}, \\
E(13) = \left\{ (-1,1,7) \right\}, E(16) = \left\{ (0,2,4) \right\}, \mbox{ and } E(17) = \left\{ (-1,1,9),(-1,3,3) \right\}.
\end{array}
\]

\section{Orbifold points by spin} \label{sec:spin}
By the results in \cite{mcmullen:spin}, the Weierstrass curve $W_D$ is usually irreducible.  For discriminants $D > 9$ with $D \equiv 1 \bmod 8$, $W_D$ has exactly two irreducible components, and they are distinguished by a spin invariant.  Throughout this section, we assume $D$ is a such a discriminant.  For such $D$, the number of points of orbifold order two is given by
\[ e_2(W_D) = \frac{1}{2} \wt h(-4D) \]
and by our Proof of Theorem \ref{thm:orbformula} in Section \ref{sec:wdorbpts}, an orbifold point of order two $(X,[\omega]) \in W_D$ is the $\rho(J)$-eigenform for a faithful $D_8$-action $\rho$ on $X$ and the quotient $E/\rho(r)$ corresponds to an ideal class $[I] \in H(\ord_{-4D})$.  In this section, we define a spin homomorphism $\epsilon_0 \co H(-4D) \rightarrow \zz/2\zz$ and relate the spin of $(X,[\omega])$ to $\epsilon_0([I])$ allowing us to establish the formula in Theorem \ref{thm:spin}.

\paragraph{One forms with double zero and spin structures.}  We start by recalling spin structures on Riemann surfaces following \cite{mcmullen:spin} (cf. also \cite{atiyah:spin}).  A {\em spin} structure on a symplectic vector space $V$ of dimension $2g$ over $\zz/2\zz$ is a quadratic form:
\[ q \co V \rightarrow \zz/2\zz \]
satisfying $q(x+y) = q(x)+q(y)+\left<x,y\right>$ where $\left<,\right>$ is the symplectic form.  The {\em parity} of $q$ is given by the {\em Arf invariant}:
\[ \Arf(q) = \sum_i q(a_i)q(b_i) \in \zz/ 2 \zz\]
where $a_1,b_1,\dots,a_g,b_g$ is a symplectic basis for $V$.

A one form with double zero $\omega$ on $X \in \M_2$ determines a spin structure on $H_1(X,\zz/2\zz)$ as follows.  For any loop $\gamma \co S^1 \rightarrow X$ whose image avoids the zero of $\omega$, gives a Gauss map:
\[ G_{\gamma} \co S^1 \rightarrow S^1 \mbox{ where } G_\gamma(x) = \omega(\gamma'(x))/\abs{\gamma'(x)}.\]  
The degree of $G_\gamma$ is invariant under homotopy that avoids the zero of $\omega$ and changes by a multiple of two under general homotopy.  Denoting by $[\gamma]$ the class in $H_1(X,\zz/2/\zz)$ represented by $\gamma$, the function
\[ q([\gamma]) = 1+\deg(G_\gamma)   \bmod 2 \]
defines a spin structure on $X$.

\paragraph{Pinwheel spin.} Now consider the surface $X_\tau$ obtained from the pinwheel $P_\tau$ and the one form $\omega_\tau \in \Omega(X_\tau)$ obtained from $dz$ on $P_\tau$ as usual.  Let $b = \tau+\frac{1-i}{2}$ and let $x_t \in H_1(X_\tau,\zz/2\zz)$ be the class represented by the integral homology class with period $t \in \zz[i] \oplus b \zz[i]$ when integrated against $\omega_\tau$.  The classes $x_1$, $x_i$, $x_b$ and $x_{bi}$ form a basis for $H_1(X_\tau,\zz / 2 \zz)$, and we have:
\begin{prop} \label{prop:spinstructure}
  The spin structure $q$ on $X_\tau$ associated to $\omega_\tau$ has:
  \begin{equation}\label{eqn:d8spin} q\left( k_1 x_1 +k_2 x_i + k_3 x_b + k_4 x_{bi} \right) = k_1^2 + k_2^2 + k_1 k_3 + k_2 k_4 + k_3 k_4 \bmod 2. \end{equation}
\end{prop}
\begin{proof}
Finding loops $\gamma_t \co S^1 \rightarrow X_\tau$ representing $x_t$ and avoiding the zero of $\omega_\tau$, we can compute the degree of the Gauss map directly and show that:
\[ q(x_1)=q(x_i) = 1 \mbox{ and } q(x_b)=q(x_{bi})=0. \]
From the relation $q(x+y)=q(x)+q(y)+\left<x,y\right>$, any basis $x_n$ of $H_1(X_\tau,\zz/2\zz)$ has:
\[ q\left(\sum_n k_n x_n\right) = \sum_n k_n^2 q(x_n) + \sum_{l<n} k_l k_n \left< x_l, x_n \right>. \]
Computing the intersection pairing on the basis $\left\{ x_1,x_i,x_b,x_{bi} \right\}$ gives the stated formula for $q$.
\end{proof}

\paragraph{Spin components of $W_D$.}  For $(X,[\omega]) \in W_D$, the vector space $H_1(X,\zz/2\zz)$ is also an $\ord_D$-module via $\iota_+^{[\omega]} \co \ord_D \rightarrow \End(\jac(X))$.  Let $f$ be the {\em conductor} of $\ord_D$, i.e. the integer satisfying $D = f^2 D_0$ where $D_0$ is the maximal order in $K_D$.  Since $D \equiv 1 \bmod 8$, the image $W$ of $(f+\sqrt{D})/2$ acting on $H_1(X,\zz/2\zz)$ is two dimensional, and we define the spin of $(X,[\omega])$ to be the parity of $q$ restricted to $W$:
\[ \epsilon(X,[\omega]) = \Arf(q|_W). \]
We also define:
\[ W_D^i = \left\{ (X,[\omega]) \in W_D : \epsilon(X,[\omega]) = i \right\}, \mbox{ for } i = 0\mbox{ or }1.\]
As was shown in \cite{mcmullen:spin}, for discriminants $D>9$ with $D \equiv 1 \bmod 8$, the components $W_D^0$ and $W_D^1$ are both non-empty and irreducible.

\paragraph{Ideal classes.}  Now let $I \subset K_{-4D}$ be a fractional and proper $\ord_{-4D}$-ideal, i.e. satisfies $\End(I) = \ord_{-4D}$, and let $E_I = \cc/I$.  There are three $D_8$-surfaces $(X,\rho)$ with $X/\rho(r) = E_I$.  From Table \ref{tab:orders}, we see that exactly one of these, which we will call $(X_I,\rho_I)$, is labeled by an orbifold point on $W_D$, with the others being labeled by orbifold points on $W_{16 D}$.  This $D_8$-surface satisfies $g(X_I,\rho_I) = (E_I,Z,T)$ where $Z = 0 + I$ and $T$ generates a subgroup of $E[2]$ invariant under $\ord_{-4D}$.

\paragraph{Spin homomorphism.}  Now let $n$ be the odd integer satisfying $\nm(I) = 2^k n$ and define:
\begin{equation}
\label{eqn:idealspin}
\epsilon_0(I) = \frac{n-1}{2} \bmod 2.
\end{equation}
We will give a formula for the spin invariant of the orbifold point on $W_D$ corresponding to $(X_I,\rho_I)$ in terms of $\epsilon_0(I)$.  To start, we show:
\begin{prop}
  \label{prop:spinhomomorphism}
  The number $\epsilon_0(I)$ depends only on the ideal class of $I$ and defines a spin homomorphism:
  \[ \epsilon_0 \co H(-4D) \rightarrow \zz/2\zz. \]
  The spin homomorphism $\epsilon_0$ is the zero map if and only if $D$ is a square.
\end{prop}
\begin{proof}
  Any $x \in \order_{-4D}$ has norm $\nm(x) = x_1^2 + x_2^2 D = 2^k l$ with $l \equiv 1 \bmod 4$.  If the ideals $I$ and $J$ are in the same ideal class, they satisfy $xI = yJ$ for some $x$ and $y$ in $\ord_{-4D}$ and $\epsilon_0(I) = \epsilon_0(J)$.  The map $\epsilon_0$ is a homomorphism since the norm of ideals is a homomorphism.  

Now suppose $D=f^2$ is a square.  Any $\ord_{-4D}$-ideal class has a representative of the form $I=x\zz\oplus(fi-y)\zz$ with $x$ and $y$ in $\zz$.  Since $I$ is an ideal, $x$ divides $f^2+y^2$ and since $I$ is proper $\gcd(x,y,(f^2+y^2)/x)=1$.  If an odd prime $p$ divides $\nm(I)$, then $p$ divides $x^2$, $f^2+y^2$, and $y$.  Since $f^2 \equiv -y^2 \bmod p$ and $p$ does not divide both $f$ and $y$, $-1$ is a square mod $p$, $p \equiv 1 \bmod 4$ and $\epsilon_0(I)=0$.

If $D$ is not a square, $D = p_1^{k_1} p_2^{k_2} \dots p_n^{k_n}$ with $p_l$ distinct odd primes and $k_1$ odd.  By Dirichlet's theorem, there is a prime $p$ with:
\begin{itemize}
  \item $p \equiv 3 \bmod 4$,
  \item $p \equiv 1 \bmod p_l$ for $l > 1$, and
  \item $\lsymb{p}{p_1} = -1$.
\end{itemize}
Quadratic reciprocity gives $\lsymb{-D}{p} = 1$ and guarantees a solution $x$ to $x^2 \equiv -D \bmod p$.  The ideal $I=p\zz \oplus (\sqrt{-D}-x)\zz$ is an $\ord_{-4D}$-ideal, has norm $p$ and $\epsilon_0(I)=1$.
\end{proof}
\begin{rmk}
  When $D \equiv 1 \bmod 8$, the ideal $(2)$ ramifies in $\ord_{-4D}$ and there is a prime ideal $P$ with $P^2=(2)$.  Since $\nm(P)=2$, we have $\epsilon_0(P)=0$.  The ideal classes represented by $I = \zz\oplus \tau \zz$ and $J=\zz \oplus -1/2\tau \zz$ satisfy $[I]=[P J]$.  This is related to the fact the polygonal pinwheels $P_\tau$ and $P_{-1/2\tau}$ give the same point on $W_D$ and so must have the same spin invariant.
\end{rmk}

\begin{prop}
Fix a $\tau \in U$ with $(X_\tau,[\omega_\tau]) \in W_D$ and let $I = \zz \oplus \tau \zz$.  The spin of $(X_\tau,[\omega_\tau])$ is given by the formula:
\[ \epsilon(X_\tau,[\omega_\tau]) = \frac{f+1}{2} +\epsilon_0([I]) \bmod 2. \]
\end{prop}
\begin{proof}
We saw in Section \ref{sec:d8family} that $X_\tau$ has a faithful $D_8$-action $\rho_\tau$ with $g(X_\tau,\rho_\tau)  = (E_I=\cc/I,Z=0+I,T=1/2+I)$.  By our proof of Theorem \ref{thm:orbformula} in Section \ref{sec:wdorbpts}, we have that $\End(E_I) = \ord_{-4D}$ and $T$ generates a subgroup of $E_I[2]$ invariant under $\ord_{-4D}$, i.e. $1+\sqrt{-D} \in 2 I$.  

Since $I$ is an ideal, $\sqrt{-D} = x \tau + y$ for some $x$ and $y \in \zz$, $x$ divides $D+y^2$ and $I$ has the same class as $I_0=x\zz \oplus (\sqrt{-D}-y)\zz$.  Since $I$ is proper, $\gcd(x,y,(D+y^2)/x)=1$ and the norm of $I_0$ is $x$ up to a factor of two.  The condition $1 + \sqrt{-D} \in 2I$ implies that $x \equiv 2 \bmod 4$ and $\epsilon_0(I) \equiv \frac{x-2}{4} \bmod 2$.

To compute the spin invariant $\epsilon(X_\tau,[\omega_\tau])$, we need to determine the subspace $W = \Im\left(\frac{f+\sqrt{D}}{2}\right)$ of $H_1(X,\zz/2\zz)$ and evaluate $\Arf(q|_W)$.  The subspace $W$ is spanned by $v$ and $J_\tau v$ where:
\[ v = x_{(f-i\sqrt{-D})/2} \equiv \frac{2f+x}{4} x_1 +\frac{x-2y}{4} x_i +x_{bi} \bmod 2. \]
Here, as above, $x_t \in H_1(X,\zz/2\zz)$ is the homology class represented by an integral homology class with $\omega_\tau$-period $t$.  By Proposition \ref{prop:spinstructure} and the fact that $q$ is $J_\tau$-invariant, we have:
\[ \epsilon(X_\tau,[\omega_\tau]) = q(v)^2 = \frac{f+1}{2} + \epsilon_0([I]) \bmod 2. \]
\end{proof}
Our formula for the number of orbifold points on the spin components of $W_D$ follows readily from the previous two propositions:
\begin{proof}[Proof of Theorem \ref{thm:spin}]
When $D=f^2$, the spin homomorphism $\epsilon_0$ is the zero map and all of the orbifold points on $W_D$ lie on the spin $(f+1)/2 \bmod 2$ component of $W_D$.  When $D$ is not a square, $\epsilon_0$ is onto, and exactly half of the orbifold points on $W_D$ lie on each spin component.
\end{proof}

\begin{rmk}
For square discriminants $D=f^2$, there is an elementary argument that shows $e_2\left(W_D^{(f-1)/2}\right)=0$.  For $(X_\tau,[\omega_\tau]) \in W_D$, the number $\tau$ is in $\qq(i)$ and rescaling $P_\tau$, we can exhibit $X_\tau$ as the quotient of a polygon $\lambda P_\tau$ with vertices in $\zz[i]$ and area $f$.  The surface $X_\tau$ is ``square-tiled'' and admits a degree $f$ map $\phi\co X_\tau \rightarrow \cc/\zz[i]$ branched over a single point.  The spin of $(X_\tau,[\omega_\tau])$ can be determined from the number of Weierstrass points $X_\tau^W$ mapping to the branch locus of $\phi$ (\cite{mcmullen:spin}, Theorem 6.1).  This number is, in turn, determined by $f$, as can be seen by elementary Euclidean geometry.
\end{rmk}

\begin{cor}
  Fix $D\geq 9$ with $D \equiv 1 \bmod 8$ and conductor $f$, a pinwheel prototype $(e,c,b) \in E(W_D)$ and set $\tau = \frac{e+\sqrt{-D}}{2c}$.  The surface $(X_{\tau},[\omega_\tau])$ has spin given by:
  \[ \epsilon(X_\tau,[\omega_\tau]) = \frac{c+f}{2} \bmod 2. \]
\end{cor}
\begin{proof}
  The surface $X_\tau$ corresponds to the ideal class of $I = 2c \zz \oplus (-e+\sqrt{-D}) \zz$ in $H(-4D)$ and the norm of $I$ is $2c$.
\end{proof}

\section{Genus of $W_D$} \label{sec:genus}
Together with \cite{bainbridge:eulerchar,mcmullen:spin}, Theorems \ref{thm:orbformula} and \ref{thm:spin} complete the determination of the homeomorphism type of $W_D$, giving a formula for the genus of the irreducible components of $W_D$.  In this section, we will prove the following upper bound on the genus of the components of $W_D$:
\begin{prop} \label{prop:upperbounds}
  For any $\epsilon>0$, there are positive constants $C_\epsilon$ and $N_\epsilon$ such that: 
  \[C_\epsilon D^{3/2+\epsilon} > g(V)\]
  whenever $V$ is a component of $W_D$ and $D>N_\epsilon$.
\end{prop}
We will also give effective lower bounds:
\begin{prop} \label{prop:effectivebounds}
  Suppose $D > 0$ is a discriminant and $V$ is a component of $W_D$.  If $D$ is not a square, the genus of $V$ satisfies:
  \[ g(V) \geq D^{3/2}/600 - D/16  - D^{3/4}/2 - 75. \]
  If $D$ is a square, the genus of $V$ satisfies:
  \[ g(V) \geq D^{3/2}/240-7D/10-D^{3/4}/2-75. \]
\end{prop}
These two propositions immediately imply Corollary \ref{cor:genus}.  Also, the following proposition shows that the components of $\bigcup_D W_D$ with genus $g \leq 4$ are all listed in Appendix \appwdtop, giving Corollary \ref{cor:genuszerocomponents} as an immediate consequence:
\begin{cor}
  The components of $\bigcup_D W_D$ with genus $g \leq 4$ all lie on $\bigcup_{D \leq 121} W_D$.
\end{cor}
\begin{proof}
  The bounds in Proposition \ref{prop:effectivebounds} show that $g(V) > 4$ whenever $D > 7000$ for non-square $D$ and $D > 200^2$ for square $D$.  The remaining discriminants were checked by computer.
\end{proof}

\paragraph{Orbifold Euler characteristic and genus.}  Let $\Gamma \subset \PSL_2(\rr)$ be a lattice and let $X = \hh/\Gamma$ be the finite volume quotient.  The homeomorphism type of $X$ is determined by the number of cusps $C(X)$ of $X$, the number $e_n(X)$ of points of orbifold order $n$ for each $n > 1$, and the genus $g(X)$ of $X$.  The {\em orbifold Euler characteristic of $X$} is the following linear combination of these numbers:
\[ \chi(X) = 2 - 2 g(X) - C(X) - \sum_{n} (1-1/n) e_n(X) . \]

\paragraph{Euler characteristic of $X_D$ and $W_D$.}
The Hilbert modular surface $X_D$ has a meromorphic modular form with a simple zero along $W_D$ and simple pole along $P_D$.  This gives a simple relationship between the orbifold Euler characteristics of $W_D$, $P_D$ and $X_D$ and a modular curve $S_D$ in the boundary of $X_D$.  The curve $S_D$ is empty unless $D=f^2$ is a square, in which case $S_D \cong X_1(f)$.
\begin{thm}[\cite{bainbridge:eulerchar} Cor. 10.4] The Euler characteristic of $W_D$ satisfies: 
  \[ \chi(W_D) = \chi(P_D) - 2 \chi(X_D) - \chi(S_D). \]
\end{thm}
For a discriminant $D$, define:
\[ F(D) = \prod_{p | f} \left(1-\lsymb{D_0}{p} p^{-2}\right). \]
where $f$ is the conductor of $\ord_D$, $D_0=D/f^2$ is the discriminant of the maximal order in $\qq(\sqrt{D})$ and the product is over primes dividing $f$.  The number $F(D)$ satisfies $1 \geq F(D) > \zeta_\qq(2)^{-1} > 6/10$.  

For square discriminants, $\chi(S_D) = -f^2 F(D)/12$ and the Euler characteristic of $W_D$ and its components are given by (\cite{bainbridge:eulerchar} Theorem 1.4): 
\[ \begin{array}{c} 
  \chi(W_{f^2}) = -f^2(f-1)F(D)/16, \chi(W_{f^2}^0) = -f^2(f-1) F(D)/32, \\
 \mbox{ and }
  \chi(W_{f^2}^1) = -f^2(f-3) F(D)/32.
\end{array}
\]
For non-square discriminants, $\chi(S_D)=0$ and $\chi(P_D)=-\frac{5}{2}\chi(X_D)$ giving $\chi(W_D)=-\frac{9}{2}\chi(X_D)$.  The Euler characteristic $\chi(X_D)$ can be computed from (\cite{bainbridge:eulerchar} Theorem 2.12):
\[ \chi(X_D) = 2 f^3 \zeta_{D_0}(-1) F(D). \]
Here $\zeta_{D_0}$ is the Dedekind-zeta function and can be computed from Siegel's formula (\cite{bruinier:modularforms}, Cor. 1.39):
\[ \zeta_{D_0}(-1) = \frac{1}{60} \sum_{e^2 < D_0,\, e \equiv D_0 \bmod 2} \sigma \left( \frac{D_0-e^2}{4} \right), \]
where $\sigma(n)$ is the sum of the divisors of $n$.  For reducible $W_D$, the spin components satisfy $\chi(W_D^0)=\chi(W_D^1) = \frac{1}{2}\chi(W_D)$ (\cite{bainbridge:eulerchar} Theorem 1.3).

The well-known bound $\sigma(n) = o(n^{1+\epsilon})$ gives constants $C_\epsilon$ and $N_\epsilon$ so:
\[ C_\epsilon D^{3/2+\epsilon} > \chi(X_D) \]
whenever $D > N_\epsilon$.  Using $\sigma(n) > n+1$ and $F(D) > 6/10$ gives:
\[ \chi(X_D) > D^{3/2}/300. \]
We can now prove the upper bounds for the genus of $W_D$:
\begin{proof}[Proof of Proposition \ref{prop:upperbounds}]
  For square discriminants $D=f^2$, we have $\abs{\chi(W_D)} \leq f^3$.  For non-square discriminants, the bounds for $\chi(X_D)$ and the formula $\chi(W_D)=-9\chi(X_D)/2$ gives $\abs{\chi(W_D)} = O( D^{3/2+\epsilon})$.  Since $W_D$ has one or two components, $g(W_D) = O(\abs{\chi(W_D)})$.
\end{proof}

\paragraph{The modular curve $P_D$.} The modular curve $P_D$ is isomorphic to:
\[ \left( \bigsqcup_{(e,l,m)} Y_0(m) \right) / g, \]
where the union is over triples of integers $(e,l,m)$ with:
\[ \begin{array}{ccc} D = e^2 + 4 l^2m, & l,m > 0, & \mbox{and } \gcd(e,l)=1 \end{array}, \]
  and $g$ is the automorphism sending the degree $m$ isogeny $i$ on the component labeled by $(e,l,m)$ to the isogeny $i^*$ on the $(-e,l,m)$-component (cf. \cite{mcmullen:spin} Theorem 2.1).  The isogeny $i\co E \rightarrow F$ on the $(e,l,m)$-component corresponds to the Abelian variety $B=E \times F$ with $\ord_D$ generated by $\iota\left( \frac{e+\sqrt{D}}{2} \right) = i + i^*+[e]_E$ where $[e]_E$ is the multiplication by $e$-map on $E$.
 
In particular, the components of $P_D$ are labeled by triples $(e,l,m)$ as above subject to the additional condition $e \geq 0$.  We will need the following bound on the number of such triples:
\begin{prop} \label{prop:componentspd}
  The number of components of $P_D$ satisfies $h_0(P_D) \leq D^{3/4}+150$.
\end{prop}
\begin{proof}
  Let $l(n)$ denote the largest integer whose square divides $n$ and let $f(n) = d(l(n))$ be the number of divisors of $l(n)$.  The function $f$ is multiplicative and the number of triples $(e,l,m)$ with $e$ fixed is bounded above by $f\left( \frac{D-e^2}{4} \right)$.  There is a finite set $S$ of natural numbers $n$ for which $f(n) > n^{1/4}$ (they are all divisors of $2^{12} 3^6 5^4 7^2 11^2$) since $d(n)$ is $o(n^\epsilon)$ for any $\epsilon > 0$ and it is easy to check that $\sum_{n \in S} f(n) - n^{1/4} < 150$.  The asserted bound on $h^0(P_D)$ follows from:
\[ h_0(P_D) \leq \sum_{\substack{ e \equiv D \bmod 2\\ 0 \leq  e < \sqrt{D} }} f\left( \frac{D-e^2}{4} \right) \leq 150 + \sum_e \left( \frac{D-e^2}{4} \right)^{1/4}. \]
\end{proof}

\paragraph{Cusps on $W_D$ and $P_D$.}  Let $C_1(W_D)$ and $C_2(W_D)$ be the number of one- and two-cylinder cusps on $W_D$ respectively and $C(W_D)=C_1(W_D)+C_2(W_D)$ be the total number of cusps.  The cusps on $W_D$ were first enumerated and sorted by component in \cite{mcmullen:spin}:
\begin{prop} \label{prop:cuspswd}
  For non-square discriminants, the number of cusps on $W_D$ is equal to the number of cusps on $P_D$:
  \[ C(W_D)=C_2(W_D) =C(P_D), \]
  and $C(W_D^0)=C(W_D^1)$ when $W_D$ is reducible.  For square discriminants $D=f^2$, the number of one- and two-cylinder cusps satisfy:
  \[ C_2(W_{f^2}) < C(P_{f^2}) \mbox{ and } C_1(W_{f^2}) < f^2/3. \]
  When $f$ is odd, $\abs{C(W_{f^2}^1)-C(W_{f^2}^0)} < 7f^2/12$.
\end{prop}
\begin{proof}
  Except for the explicit bounds on $C_1(W_{f^2})$ and $\abs{C(W_{f^2}^1)-C(W_{f^2}^0)}$, the claims in the proposition follow from the enumeration of cusps on $P_D$ and $W_D$ in \cite{bainbridge:eulerchar}, \S 3.1.  
  
We now turn to the bounds on $C_1(W_{f^2})$ and $\abs{C(W_{f^2}^1)-C(W_{f^2}^0)}$.  When $D$ is not a square, there are no one-cylinder cusps and when $D=f^2$ is a square, the one-cylinder cusps are parametrized by cyclically ordered triples $(a,b,c)$ with (cf. \cite{mcmullen:spin} Theorem A.1):
  \[ f=a+b+c,\, a, b,c>0 \mbox{ and } \gcd(a,b,c)=1. \]
Cyclically reordering $(a,b,c)$ so $a < b$ and $a < c$ ensures that $a < f/3$ and $b < f$, giving $C_1(W_D) < f^2/3=D/3$.  The difference in the number of two cylinder cusps is given by (Theorem A.4 in \cite{mcmullen:spin}):
 \[ C_2(W_D^0) - C_2(W_D^1) = \sum_{b+c=f,0<c<b} \phi(\gcd(b,c)), \] 
 which is smaller than $D/4$ using $c < f/2$ and $\phi(\gcd(b,c)) < f/2$.  The bound asserted for $\abs{C(W_D^0)-C(W_D^1)}$ follows.
\end{proof}
We are now ready to prove the lower bounds in Proposition \ref{prop:effectivebounds}:
\begin{proof}[Proof of Proposition \ref{prop:effectivebounds}] 
We will handle the square and non-square discriminants separately.
\paragraph{Lower bounds, non-square discriminants.}
Suppose $D$ is a non-square discriminant and $D>8$ so all of the orbifold points on $W_D$ have order two.  Using the formula $\chi(W_D) = \chi(P_D) - 2\chi(X_D)$, the equality $C(P_D)=C(W_D)$ and ignoring several terms which contribute positively to the $g(W_D)$ gives:
\[ g(W_D) \geq \chi(X_D) - h^0(P_D) - e_2(W_D)/4. \]
Combining the bound above with $\chi(X_D) > D^{3/2}/300$, $h_0(P_D) < D^{3/4}+150$, $e_2(W_D) < D/2$ (Proposition \ref{prop:orbptbounds}) and $g(V) \geq \frac{1}{2} g(W_D)$ whenever $V$ is a component of $W_D$ gives the bound stated in Proposition \ref{prop:effectivebounds}.

\paragraph{Lower bounds, square discriminants.}  Now suppose $D=f^2$.  Using the formula for $\chi(W_D)$ in terms of $\chi(X_D)$, $\chi(P_D)$ and $\chi(S_D)$, the bound $C_2(W_D) < C(P_D)$ and ignoring some terms which contribute positively to $g(W_D)$ gives:
\[ g(W_D) \geq \chi(X_D) - h_0(P_D) - e_2(W_D)/4 + \chi(S_D)/2 - C_1(W_D)/2. \]
As before we have $h_0(P_D) < D^{3/4}+150$, $e_2(W_D) < D/2$ and $C_1(W_D) < D/3$.  By Theorem 2.12 and Proposition 10.5 of \cite{bainbridge:eulerchar} and using $\zeta_\qq(2) > 6/10$, we have $\chi(X_D)+\chi(S_D)/2 > D^{3/2}/120-D/40$ so long as $D>36$, giving:
\[ g(W_D) \geq D^{3/2}/120-2 D/5 - D^{3/4}-150. \]

Finally, to bound $g(V)$ when $V$ is a component of $W_D$, we bound the difference:
\[ 
\begin{array}{ll}
 \abs{g(W_D^0) - g(W_D^1)} &\leq \abs{\frac{\chi(W_D^1)-\chi(W_D^0)}{2}} + \abs{\frac{C(W_D^1)-c(W_D^0)}{2}}+e_2(W_D)/4 \\
\end{array}.
\]
We have seen that $\abs{C(W_D^1) - C(W_D^0)} < 7D/12$ and $e_2(W_D)/4 < D/8$.  Theorem 1.4 of \cite{bainbridge:eulerchar} gives $\abs{\chi(W_D^0)-\chi(W_D^1)}< D/16$ and $\abs{g(W_D^1)-g(W_D^0)} < D/2$. The bound asserted for $g(V)$ in Proposition \ref{prop:effectivebounds} follows. \end{proof}

\appendix
\section{The $D_{12}$-family}
\label{sec:d12family} In this section we will describe the surfaces in $\M_2(D_{12})$.  For a smooth surface $X \in \M_2$, the following are equivalent:
\begin{itemize}
  \item {\em Automorphisms.}  The automorphism group $\Aut(X)$ admits an injective homomorphism $\rho: D_{12} \rightarrow \Aut(X)$.
  \item {\em Algebraic curves.}  The field of functions $\cc(X)$ is isomorphic to the field:
    \[ \wt K_a = \cc(z,x) \mbox{ with } z^2=x^6-ax^3+1, \]
    for some $a \in \cc \setminus \left\{ \pm 2 \right\}$.
  \item {\em Jacobians.}  The Jacobian $\jac(X)$ is isomorphic to the principally polarized Abelian variety:
    \[ \wt A_\tau = \cc^2 / \wt \Lambda_\tau, \]
    where $\wt \Lambda_\tau = \zz \left< \spmat{1 \\ 1/\sqrt{3}}, \spmat{\tau \\ \sqrt{3} \tau}, \spmat{1 \\ -1/\sqrt{3}}, \spmat{ \tau \\ -\sqrt{3} \tau} \right>$ and is polarized by the symplectic form $\left< \spmat{a \\ b}, \spmat{c \\ d} \right> = \frac{-\Im(a \overline c + b \overline d)}{2 \Im \tau} $.
  \item {\em Hexagonal pinwheels.}  The surface $X$ is isomorphic to the surface $\wt X_\tau$ obtained by gluing the hexagonal pinwheel $H_\tau$ (Fig. \ref{fig:hexpinwheels}) to $-H_\tau$ for some $\tau$ in the domain:
    \[ \wt U =  \left\{ \tau \in \hh : \tau \neq \zeta_{12}/\sqrt{3} \mbox{ or } \zeta_{12}^5/\sqrt{3}, \abs{\Re{\tau}} \leq \frac{1}{2} \mbox{ and } \abs{\tau}^2 \geq \frac{1}{3} \right\}. \]

\end{itemize}
It is straightforward to identify the action of $D_{12}$ on the surfaces described above.  The field $\wt K_a$ has automorphisms $Z(z,x) = (-z,\zeta_3x)$ and $r(z,x)=(z/x^3,1/x)$.  The polarized lattice $\wt \Lambda_\tau$ is preserved by the linear transformations $r=\spmat{1 & 0 \\ 0 & -1}$ and $Z=\frac{1}{2}\spmat{1 & -\sqrt{3} \\ \sqrt{3} & 1}$.  The surface obtained from $\wt X_\tau$ has an order six automorphism $Z_\tau$ with $[Z_\tau|_{\wt X_\tau^W}] = [3,3]$ which implies that $\wt X_\tau$ has a faithful $D_{12}$-action (cf. Table \ref{tab:g2auts} in \S \ref{sec:d8family}).

The family $\M_2(D_{12})$ admits an analysis similar to that of $\M_2(D_8)$.  For $(X,\rho) \in \M_2(D_{12})$, the quotient $E = X/\rho(r)$ has genus one and a distinguished subgroup of order three in $E[3]$.  One can establish the precise relationship between $\wt X_\tau$, $\wt A_\tau$ and $\wt K_a$ by studying the corresponding map from $\M_2(D_{12})$ to the modular curve $Y_0(3)$.

\begin{figure}
  \begin{center}
    \includegraphics[scale=0.25]{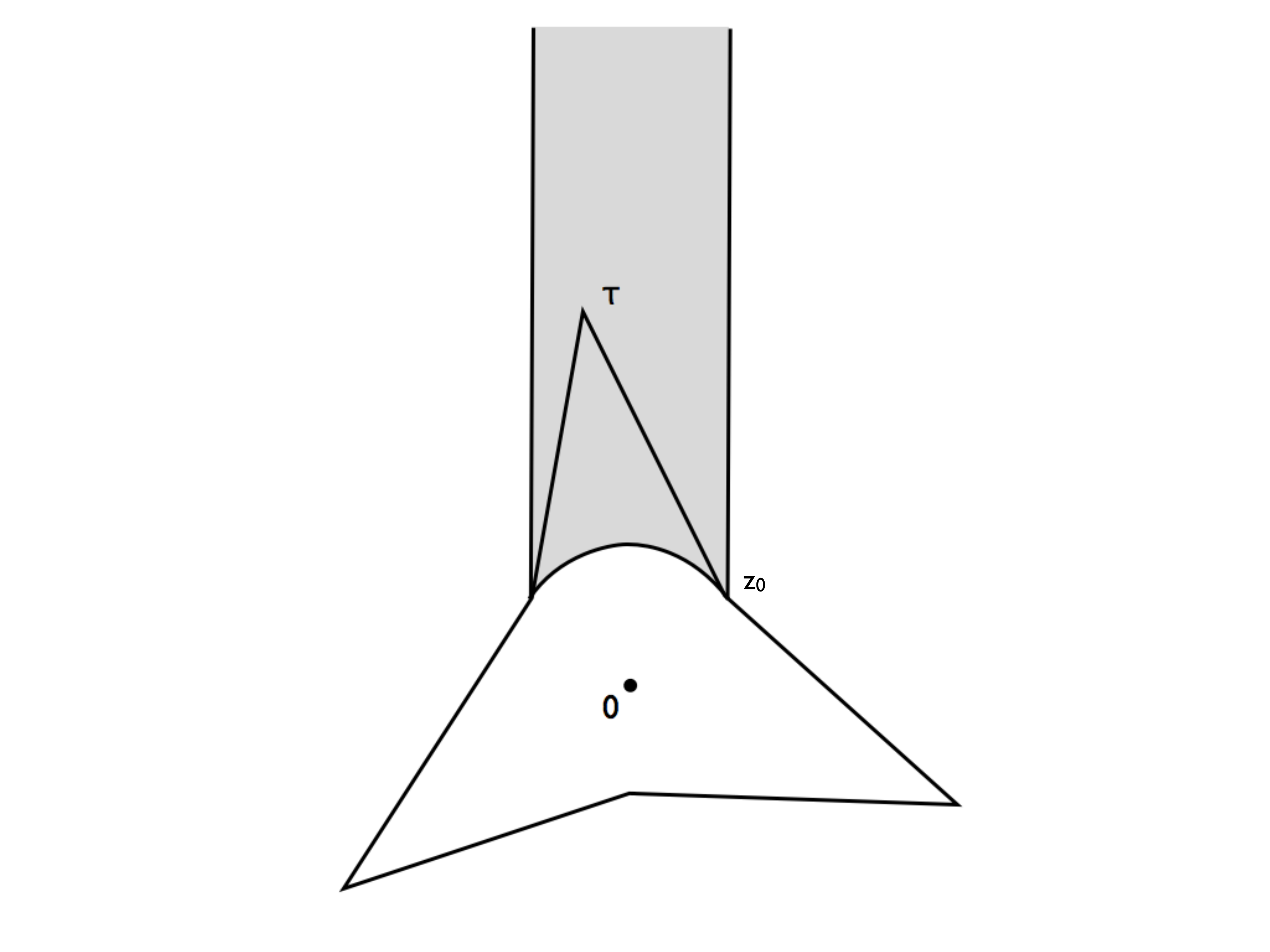}
  \end{center}
  \caption{\label{fig:hexpinwheels} {\sf For $\tau$ in the shaded domain $\wt U$, the hexagonal pinwheel $H_\tau$ has vertices $\left\{ z_0,\zeta_3^{\pm1} z_0, \tau, \zeta_3^{\pm 1} \tau \right\}$ with $z_0=\zeta_{12}/\sqrt{3}$.  Gluing together sides on $H_\tau$ and $-H_\tau$ by translation gives a genus two surface admitting an action of $D_{12}$.  The one form induced by $dz$ is a $Z$-eigenform.}}
\end{figure}

\newpage

\begin{center}
\begin{table}
\Large
{\bf \appwdtop \hspace{0.1in} Homeomorphism type of $W_D$}
\addcontentsline{toc}{section}{\appwdtop \hspace{0.15in} Homeomorphism type of $W_D$}
 \footnotesize
\[
\begin{array}{cc}
  \begin{array}{|c|cccc|}
  \hline
  \rule{0pt}{12pt}
    D  & g(W_D) & e_2(W_D) & C(W_D) & \,\chi(W_D)\, \\
 \hline
 \rule{0pt}{12pt}
 5 & 0 & 1 & 1 & -\frac{3}{10} \\
 8 & 0 & 0 & 2 & -\frac{3}{4} \\
 9 & 0 & 1 & 2 & -\frac{1}{2} \\
 12 & 0 & 1 & 3 & -\frac{3}{2} \\
 13 & 0 & 1 & 3 & -\frac{3}{2} \\
 16 & 0 & 1 & 3 & -\frac{3}{2} \\
 17 & \{0,0\} & \{1,1\} & \{3,3\} & \left\{-\frac{3}{2},-\frac{3}{2}\right\} \\
 20 & 0 & 0 & 5 & -3 \\
 21 & 0 & 2 & 4 & -3 \\
 24 & 0 & 1 & 6 & -\frac{9}{2} \\
 25 & \{0,0\} & \{0,1\} & \{5,3\} & \left\{-3,-\frac{3}{2}\right\} \\
 28 & 0 & 2 & 7 & -6 \\
 29 & 0 & 3 & 5 & -\frac{9}{2} \\
 32 & 0 & 2 & 7 & -6 \\
 33 & \{0,0\} & \{1,1\} & \{6,6\} & \left\{-\frac{9}{2},-\frac{9}{2}\right\} \\
 36 & 0 & 0 & 8 & -6 \\
 37 & 0 & 1 & 9 & -\frac{15}{2} \\
 40 & 0 & 1 & 12 & -\frac{21}{2} \\
 41 & \{0,0\} & \{2,2\} & \{7,7\} & \{-6,-6\} \\
 44 & 1 & 3 & 9 & -\frac{21}{2} \\
 45 & 1 & 2 & 8 & -9 \\
 48 & 1 & 2 & 11 & -12 \\
 49 & \{0,0\} & \{2,0\} & \{10,8\} & \{-9,-6\} \\
 \hline
 \end{array}
 & 
 \begin{array}{|c|cccc|}
  \hline
  \rule{0pt}{12pt}
 D & g(W_D) & e_2(W_D) & C(W_D) & \chi(W_D) \\
 \hline
 \rule{0pt}{12pt}
 52 & 1 & 0 & 15 & -15 \\
 53 & 2 & 3 & 7 & -\frac{21}{2} \\
 56 & 3 & 2 & 10 & -15 \\
 57 & \{1,1\} & \{1,1\} & \{10,10\} & \left\{-\frac{21}{2},-\frac{21}{2}\right\} \\
 60 & 3 & 4 & 12 & -18 \\
 61 & 2 & 3 & 13 & -\frac{33}{2} \\
 64 & 1 & 2 & 17 & -18 \\
 65 & \{1,1\} & \{2,2\} & \{11,11\} & \{-12,-12\} \\
 68 & 3 & 0 & 14 & -18 \\
 69 & 4 & 4 & 10 & -18 \\
 72 & 4 & 1 & 16 & -\frac{45}{2} \\
 73 & \{1,1\} & \{1,1\} & \{16,16\} & \left\{-\frac{33}{2},-\frac{33}{2}\right\} \\
 76 & 4 & 3 & 21 & -\frac{57}{2} \\
 77 & 5 & 4 & 8 & -18 \\
 80 & 4 & 4 & 16 & -24 \\
 81 & \{2,0\} & \{0,3\} & \{16,14\} & \left\{-18,-\frac{27}{2}\right\} \\
 84 & 7 & 0 & 18 & -30 \\
 85 & 6 & 2 & 16 & -27 \\
 88 & 7 & 1 & 22 & -\frac{69}{2} \\
 89 & \{3,3\} & \{3,3\} & \{14,14\} & \left\{-\frac{39}{2},-\frac{39}{2}\right\} \\
 92 & 8 & 6 & 13 & -30 \\
 93 & 8 & 2 & 12 & -27 \\
 96 & 8 & 4 & 20 & -36 \\
 \hline
 \end{array}
 \end{array}
 \] 
\caption{\label{tab:topwd} {\sf The Weierstrass curve $W_D$ is a finite volume hyperbolic orbifold and for $D>8$ its homeomorphism type is determined by the genus $g(W_D)$, the number of orbifold points of order two $e_2(W_D)$, the number of cusps $C(W_D)$ and the Euler characteristic $\chi(W_D)$.  The values of these topological invariants are listed for each curve $W_D$ with $D \leq 225$ as well as several larger discriminants.  When $D >9$ with $D \equiv 1 \bmod 8$, the curve $W_D$ is reducible and the invariants are listed for both spin components with the invariants for $W_D^0$ appearing first.}}

\end{table}

\begin{table}
\footnotesize
\[
  \begin{array}{|c|cccc|}
  \hline
  \rule{0pt}{12pt}
 D & g(W_D) & e_2(W_D) & C(W_D) & \chi(W_D) \\
 \hline
 \rule{0pt}{12pt}
 97 & \{4,4\} & \{1,1\} & \{19,19\} & \left\{-\frac{51}{2},-\frac{51}{2}\right\} \\
100 & 4 & 0 & 30 & -36 \\
 101 & 6 & 7 & 15 & -\frac{57}{2} \\
 104 & 9 & 3 & 20 & -\frac{75}{2} \\
 105 & \{6,6\} & \{2,2\} & \{16,16\} & \{-27,-27\} \\
 108 & 10 & 3 & 21 & -\frac{81}{2} \\
 109 & 8 & 3 & 25 & -\frac{81}{2} \\
 112 & 10 & 2 & 29 & -48 \\
 113 & \{6,6\} & \{2,2\} & \{16,16\} & \{-27,-27\} \\
 116 & 11 & 0 & 25 & -45 \\
 117 & 10 & 4 & 16 & -36 \\
 120 & 16 & 2 & 20 & -51 \\
 121 & \{6,3\} & \{3,0\} & \{26,26\} & \left\{-\frac{75}{2},-30\right\} \\
 124 & 15 & 6 & 29 & -60 \\
 125 & 11 & 5 & 15 & -\frac{75}{2} \\
 128 & 13 & 4 & 22 & -48 \\
 129 & \{8,8\} & \{3,3\} & \{22,22\} & \left\{-\frac{75}{2},-\frac{75}{2}\right\} \\
 132 & 15 & 0 & 26 & -54 \\
 133 & 15 & 2 & 22 & -51 \\
 136 & 17 & 2 & 36 & -69 \\
 137 & \{9,9\} & \{2,2\} & \{19,19\} & \{-36,-36\} \\
 140 & 19 & 6 & 18 & -57 \\
 141 & 18 & 4 & 18 & -54 \\
 144 & 11 & 4 & 38 & -60 \\
 145 & \{10,10\} & \{2,2\} & \{29,29\} & \{-48,-48\} \\
 148 & 20 & 0 & 37 & -75 \\
 149 & 16 & 7 & 19 & -\frac{105}{2} \\
 152 & 22 & 3 & 18 & -\frac{123}{2} \\
 153 & \{10,10\} & \{2,2\} & \{26,26\} & \{-45,-45\} \\
 156 & 25 & 8 & 26 & -78 \\
 157 & 20 & 3 & 25 & -\frac{129}{2} \\
 160 & 22 & 4 & 40 & -84 \\
 161 & \{14,14\} & \{4,4\} & \{20,20\} & \{-48,-48\} \\
 164 & 20 & 0 & 34 & -72 \\
 \hline
 \end{array}
 \] 
\end{table}

\begin{table}
\footnotesize
\[
  \begin{array}{|c|cccc|}
  \hline
  \rule{0pt}{12pt}
 D & g(W_D) & e_2(W_D) & C(W_D) & \chi(W_D) \\
 \hline
 \rule{0pt}{12pt}
 165 & 24 & 4 & 18 & -66 \\
 168 & 29 & 2 & 24 & -81 \\
 169 & \{14,7\} & \{0,3\} & \{37,39\} & \left\{-63,-\frac{105}{2}\right\} \\
 172 & 29 & 3 & 37 & -\frac{189}{2} \\
 173 & 22 & 7 & 13 & -\frac{117}{2} \\
 176 & 27 & 6 & 29 & -84 \\
 177 & \{17,17\} & \{1,1\} & \{26,26\} & \left\{-\frac{117}{2},-\frac{117}{2}\right\} \\
 180 & 28 & 0 & 36 & -90 \\
 181 & 26 & 5 & 33 & -\frac{171}{2} \\
 184 & 37 & 2 & 38 & -111 \\
 185 & \{17,17\} & \{4,4\} & \{23,23\} & \{-57,-57\} \\
 188 & 31 & 10 & 19 & -84 \\
 189 & 27 & 6 & 26 & -81 \\
 192 & 31 & 4 & 34 & -96 \\
 193 & \{19,19\} & \{1,1\} & \{37,37\} & \left\{-\frac{147}{2},-\frac{147}{2}\right\} \\
 196 & 25 & 0 & 60 & -108 \\
 197 & 26 & 5 & 21 & -\frac{147}{2} \\
 200 & 31 & 3 & 36 & -\frac{195}{2} \\
 201 & \{20,20\} & \{3,3\} & \{34,34\} & \left\{-\frac{147}{2},-\frac{147}{2}\right\} \\
 204 & 38 & 6 & 40 & -117 \\
 213 & 36 & 4 & 18 & -90 \\
 216 & 38 & 3 & 46 & -{243}{2} \\
 217 & \{25,25\} & \{2,2\} & \{38,38\} & \{-87,-87\} \\
 220 & 46 & 8 & 44 & -138 \\
 221 & 32 & 8 & 30 & -96 \\
 224 & 42 & 8 & 34 & -120 \\
 225 & \{21,16\} & \{4,0\} & \{42,42\} & \{-84,-72\} \\
 41376& 164821& 112& 1552& -331248 \\
 41377& \{113276, 113276 \}& \{28, 28\}& \{1442, 1442\}& \{-228006, -228006\} \\
 41380& 178100& 0& 3154& -359352\\
 41381& 119380& 89& 665& -\frac{478935}{2}\\
 41384& 145957& 68& 884& -292830\\
 41385& \{107869, 107869\}& \{24, 24\}& \{1284, 1284\}& \{-217032, -217032\}\\
 41388& 155386& 54& 1188& -311985\\ \hline
 \end{array}
\]
\end{table}
\end{center}
\normalsize

\begin{center}
\begin{table}
  \Large {\bf \apporbptmodels \hspace{0.1in} Algebraic curves labeled by orbifold points on $W_D$} \small
  \[
  \begin{array}{|c|c|}
    \hline
    D & f_D(t) \\
    \hline 
    \rule{0pt}{15pt}
    5 & t^2-68t+124\\
    8 & t+6 \\
    9 & t^2-772 t-1532\\
    12 & (t-14) (t+1) \\
    13 & t^2-5188 t -10364\\
    16 & (2t^2+73t +170) \\
    17 & t^4  - 26376 t^3 - 209384 t^2 - 943136 t  -1259504 \\
    21 & t^4 - 111752 t^3  + 555288 t^2 + 1774048 t +433168 \\
    24 & t^2+140 t +292 \\
    25 & t^2- 414724 t-829436 \\
    28 & (t-254) (16 t+31) (16 t^2+17 t+226) \\
    29 & \begin{array}{c} t^6-1390796 t^5-35420996 t^4 \\ -640534176 t^3-3448572688 t^2-7486135488 t-5780019136 \end{array} \\
    32 & (t^2 + 452 t - 124) (4 t^2 - 12 t - 41) \\
    33 & t^4-4301576 t^3+93537816 t^2+356944864 t+305326096 \\
    37 & t^2-12446788 t-24893564 \\
    40 & t^2+1292 t+2596 \\
    41 & \begin{array}{c} t^8-34052624 t^7-1944255376 t^6-98991188416 t^5 \\ -478185515936 t^4-1414176696064 t^3-4859849685248 \\ t^2-10349440893952 t-7969572716288 \end{array} \\
    44 & (t^3 - 2090 t^2 - 7604 t - 10936) (t^3 + 3 t^2 + 131 t + 257) \\
    45 & t^4-88796296 t^3+237562136 t^2+595063264 t-470492144 \\
    48 &  (t^2 + 3332 t + 10756) (16 t^2 + 272 t + 481) \\
    49 & t^4-222082568 t^3-2565706728 t^2-11151157280 t-13816147952 \\
    53 & \begin{array}{c} t^6-535413964 t^5-72289563460 t^4-9219442091680 t^3 \\ -54695502924560 t^2-110556205489344 t-74946436241344 \end{array} \\
    56 & t^4+7960 t^3-3368 t^2+18272 t+113936 \\
\hline
  \end{array}
  \] 
  \caption{\label{tab:minpoly} {\sf When $a$ is a root of $f_D(t)$, the algebraic curve $X$ satisfying $\cc(X) \cong \cc(x,z)$ where $z^2=(x^2-1)(x^4-ax^2+1)$ is labeled by a point on $W_D$.}}
\end{table}
\end{center}

\newpage

\bibliography{main}
\bibliographystyle{math}


\end{document}